\documentclass[11pt]{article}
\usepackage{amsmath, latexsym, amsfonts, amssymb, amsthm, amscd}
\usepackage[left=2cm, right=2cm, top=2.5cm, bottom=2.5cm]{geometry}
\usepackage{upquote}
\usepackage{color}
\usepackage{setspace}
\usepackage[labelfont=bf]{caption}
\usepackage{stmaryrd}

\usepackage[utf8]{inputenc}
\usepackage[T1]{fontenc}
\usepackage[english]{babel}
\usepackage{dsfont}
\usepackage{mathtools} 
\usepackage{bm}

\usepackage{xcolor}
\definecolor{Maroon}{HTML}{ad2231}
\definecolor{webgreen}{HTML}{008000}

\usepackage{hyperref}
\hypersetup{colorlinks, breaklinks, urlcolor=Maroon, linkcolor=Maroon, citecolor=webgreen} 

\allowdisplaybreaks

\usepackage{bookmark}



\newtheorem{corollary}{Corollary}
\newtheorem{proposition}{Proposition}
\newtheorem{lemma}{Lemma}

\newtheorem{theorem}{Theorem}

\theoremstyle{definition}
\newtheorem{Condition}{Condition}
\newtheorem*{general*}{A general remark}

\begin{document}
\title{The asymptotic distribution of cluster sizes for supercritical percolation
on random split trees}
\author{Gabriel Berzunza\footnote{ E-mail:
        \href{mailto:gabriel.berzunza-ojeda@math.uu.se}{gabriel.berzunza-ojeda@math.uu.se}},\,\,
    and
    Cecilia Holmgren\footnote{ E-mail: \href{mailto:cecilia.holmgren@math.uu.se}{cecilia.holmgren@math.uu.se}} \\ \vspace*{10mm}
{\small Department of Mathematics, Uppsala University, Sweden} }
\maketitle

\vspace{0.1in}

\begin{abstract} 
We consider the model of random trees introduced by Devroye (1999), the so-called random split trees. The model encompasses many important
randomized algorithms and data structures. We then perform supercritical Bernoulli bond-percolation on those trees and obtain a precise weak limit theorem for the sizes of the largest clusters. We also show that the approach developed in this work may be useful for studying percolation on other classes of trees with logarithmic height, for instance, we also study the case of $d$-regular trees.
\end{abstract}



\section{Introduction}

We investigate the asymptotic behaviour of the sizes of the largest clusters created by performing Bernoulli bond-percolation on random split trees. Split trees were first introduced by Devroye \cite{Luc1999} to encompass many families of trees that are frequently used to model efficient data structures or sorting algorithms (we will be more precise shortly). Important examples of split trees are binary search trees \cite{Hoa1962}, $m$-ary search trees \cite{Pyke1965}, quad trees \cite{finkel1974}, median-of-$(2k+1)$ trees \cite{Walker1976}, fringe-balanced trees \cite{Luc1993}, digital search trees \cite{Coffman1970} and random simplex trees \cite[Example 5]{Luc1999}. 

To be more precise, we consider trees $T_{n}$ of large but finite size $n \in \mathbb{N} \coloneqq \{1, 2, \dots\}$ and perform Bernoulli bond-percolation with parameter $p_{n} \in [0,1]$ that depends on the size of the tree (i.e., one removes each edge in $T_{n}$ with probability $1 - p_{n}$, independently of the other edges, inducing a partition of the set of vertices into connected clusters). In particular, we are going to be interested in the supercritical regime, in the sense that with high probability, there exists a giant cluster, that is of a size comparable to that of the entire tree. 

Bertoin \cite{Be1} established a simple characterization of tree families with $n$ vertices and percolation regimes which results in giant clusters. Roughly speaking, Bertoin \cite{Be1} showed that the supercritical regime corresponds to percolation parameters of the form $1-p_{n} = c/\ell(n) + o(1/\ell(n))$ as $n \rightarrow \infty$, where $c > 0$ is fixed and $\ell(n)$ is an approximation of the height of a typical vertex in the tree structure\footnote{For two sequences of real numbers $(A_{n})_{n \geq 1}$ and $(B_{n})_{n \geq 1}$ such that $B_{n} > 0$, we write $A_{n} = o(B_{n})$ if $\lim_{n \rightarrow \infty} A_{n}/B_{n} = 0$.}. Then the size $\Gamma_{n}$ of the cluster containing the
root satisfies $\lim_{n \rightarrow \infty} n^{-1} \Gamma_{n} = \Gamma(c)$ in distribution for some random variable satisfying $\mathbb{P}( \Gamma(c) = 0) < 1$.
In several examples the supercritical percolation parameter satisfies
\begin{eqnarray} \label{eq1}
p_{n} = 1 - c/\ln n + o\left( 1/\ln n \right),
\end{eqnarray}

\noindent for some fixed parameter $c >0$. For example, this happens for some important families of random trees with logarithmic height, such as random recursive trees, preferential attachment trees, binary search trees; see \cite{Drmota20092}, \cite[Section 4.4]{Durrett2010}. In those cases the random variable $\Gamma(c)$ is an (explicit) constant and the giant cluster is unique. 

A natural problem in this setting is then to estimate the sizes of the next largest clusters.
Concerning trees with logarithmic height, Bertoin \cite{Be3} proved that in the supercritical regime, the sizes of the next largest clusters of a uniform random recursive tree, normalized by a factor $\ln n /n$, converge to the atoms of some Poisson process; see also \cite{Baur2016}. This result was extended by Bertoin and Bravo \cite{Uribe2015} to preferential attachment trees. A different example is the uniform Cayley trees where $\ell(n) = \sqrt{n}$ and $\Gamma(c)$ is not constant. But unlike the previous examples, the number of giant components is unbounded as $n \rightarrow \infty$; see \cite{Pitman1999} and \cite{Pitman2006}. 

As a motivation, it is important to point out that supercritical Bernoulli bond-percolation on large but finite connected graphs is an ongoing subject of research in statistical physics and mathematics. Furthermore, the estimation of the sizes of the next largest clusters is a relevant question in this setting. An important example where the graph is not a tree is the case of a complete graph with $n$ vertices. A famous result due to Erd\"os and R\'enyi (see \cite{Bollo2001}) shows that Bernoulli bond-percolation with parameter $p_{n} = c/n + o(1/n)$ for $c >1$ fixed, produces with high probability as $n \rightarrow \infty$, a unique giant cluster of size close to $\theta(c)n$, where $\theta(c)$ is the unique solution to the equation $x + e^{-cx} =1$, while the second, third, etc.\ largest clusters have only sizes of order $\ln n$ (note that bond-percolation with parameter $p_{n}$ in the complete graph corresponds to the well-known binomial random graph $G(n, p_{n})$.) 

The main purpose of this work is to investigate the case of random split trees which belong to the family of random trees with logarithmic heights; see Devroye \cite{Luc1999}. Informally speaking, a random split tree $T_{n}^{{\rm sp}}$ of ``size'' (or cardinality) $n$ is constructed
by first distributing $n$ balls (or keys) among the vertices of an infinite $b$-ary tree ($b \in \mathbb{N}$) and then removing all sub-trees without balls. Each vertex in the infinite $b$-ary tree is given a random non-negative split vector $\mathcal{V} = (V_{1}, \dots, V_{b})$ such that $\sum_{i=1}^{b}V_{i} =1$ and $V_{i} \geq 0$, are drawn independently from the same distribution. These vectors affect how balls are distributed. Its exact definition is somewhat lengthy and we postpone it to  Section \ref{sec1}. An important peculiarity is that the number $N^{(n)}$ of vertices of $T_{n}^{{\rm sp}}$ is often random which makes the study of split trees usually challenging. 

Recently, we have shown in \cite[Lemma 2 and Lemma 3]{Ber2019} that the supercritical percolation regime in split trees of cardinality $n$ corresponds precisely to parameters fulfilling (\ref{eq1}). Note that here $n$ corresponds to the number of balls (or keys) and not to the number of vertices.  More precisely, let $C^{0}_{n}$ (resp.\ $\hat{C}^{0}_{n}$) be the number of balls (resp.\ number of vertices) in the percolation cluster that contains the root. Then, in the regime (\ref{eq1}) and under some mild conditions on the split tree (see Section \ref{sec1}), it holds that
\begin{eqnarray}
n^{-1} C^{0}_{n} \overset{\mathbb{P}}{\longrightarrow} e^{-c/\mu} \hspace*{6mm} \text{\Big(resp.} \hspace*{2mm} n^{-1} \hat{C}^{0}_{n} \overset{\mathbb{P}}{\longrightarrow} \alpha e^{-c/\mu} \Big), \hspace*{5mm} \text{as} \hspace*{2mm} n \rightarrow \infty,
\end{eqnarray}

\noindent where $\mu = b\mathbb{E}[-V_{1}\ln V_{1}]$ ($\alpha >0$ is some constant depending on the split tree) and $\overset{\mathbb{P}}{\longrightarrow}$ denotes convergence in probability. Furthermore, the giant cluster is unique. Indeed, the above results agree with \cite[Theorem 1]{Be1} even when the number of vertices in split trees is random and the cluster sizes can be defined as either the number of balls or the number of vertices. 

In this work, we extend these results and show that in the supercritical regime (\ref{eq1}) the next largest clusters of a split tree $T_{n}^{{\rm sp}}$ have a size of order $n/\ln n$. Moreover, we obtain a limit theorem in terms of a certain Poisson process. More precise statements will be given in Theorems \ref{Teo1}, \ref{Teo2} and \ref{Teo4} below. These results exhibit that cluster sizes, in the supercritical regime, of split-trees, uniform recursive trees and preferential attachment trees have similar asymptotic behaviour. Finally, we show that our present approach also applies to study the size of the largest clusters for percolation on complete regular trees (see Theorem \ref{Teo3}).

The approach developed in this work differs from that used to study the cases of uniform random recursive trees (RRT) in \cite{Be3} and preferential attachment trees in \cite{Uribe2015}. The method of \cite{Be3} is based on a coupling of Iksanov and M\"ohle \cite{Iksanov2007}  connecting the Meir and Moon \cite{Meir1970} algorithm for the isolation of
the root in an RRT and a certain random walk. This makes use of special properties of recursive trees (the so-called randomness preserving property, i.e., if one removes an edge from an RRT, then the two resulting subtrees, conditionally on their sizes, are independent RRT’s) which fail for split-trees. The basic idea of \cite{Uribe2015} is based on the close relation of preferential attachment trees with Markovian branching processes and the dynamical incorporation of percolation as neutral
mutations. The recent work of Berzunza \cite{Berzunza2018}
shows that one can also relate percolation on some types of split trees (but not all) with general age-dependent branching processes (or Crump-Mode-Jagers processes) with neutral mutations. However, the lack of the Markov property in those general branching processes makes the idea of \cite{Uribe2015} difficult to implement.

A common feature in these previous works, namely \cite{Be3} and \cite{Uribe2015}, is that, even though one addressed a static problem, one can consider a dynamical version in which edges are removed, respectively vertices are inserted, one after the other in a certain order as time passes. Here we use a fairly different route and view percolation on split trees as a static problem.

We next introduce formally the family of random split trees and relevant background, which will enable us to state our main results in Section \ref{sec2}.

\subsection{Random split trees} \label{sec1}

In this section, we introduce the split tree generating algorithm with parameters $b, s, s_{0}, s_{1}, \mathcal{V}$ and $n$ introduced by Devroye \cite{Luc1999}. Some of the parameters are the branch factor $b \in \mathbb{N}$, the vertex capacity $s \in \mathbb{N}$, and the number of balls (or cardinality) $n \in \mathbb{N}$. The additional integers $s_{0}$ and $s_{1}$ are needed to describe the ball distribution process. They satisfy the inequalities
\begin{eqnarray*} 
0 < s, \hspace*{3mm} 0 \leq s_{0} \leq s, \hspace*{3mm} 0 \leq b s_{1} \leq s + 1 - s_{0}.
\end{eqnarray*}

\noindent The so-called random split
vector $\mathcal{V} = (V_{1}, \dots, V_{b})$ is a random non-negative vector with $\sum_{i=1}^{b} V_{i} = 1$ and $V_{i} \geq 0$, for $i=1, \dots, b$. 

Consider an infinite rooted $b$-ary tree $\mathbb{T}$, i.e., every vertex has $b$ children. We view each vertex of $\mathbb{T}$ as a bucket with capacity $s$ and we assign to each vertex $u \in \mathbb{T}$ an independent copy $\mathcal{V}_{u} = (V_{u,1}, \dots, V_{u,b})$ of the random split vector $\mathcal{V}$.  Let $C(u)$ denote the number of balls in vertex $u$, initially setting $C(u) = 0$ for all $u$. We call $u$ a leaf if $C(u) > 0$ and $C(v) = 0$ for all children $v$ of $u$, and internal if $C(v) > 0$ for some strict descendant $v$ of $u$. The split tree $T_{n}^{{\rm sp}}$ is constructed recursively by distributing $n$ balls one at time to generate a subset of vertices of $\mathbb{T}$. The balls are labeled using the set $\{1, 2, \dots, n\}$ in the order of insertion. The $j$-th ball is added by the following procedure.
 \begin{enumerate}
 \item Insert $j$ to the root.
 
 \item While $j$ is at an internal vertex $u \in \mathbb{T}$, choose child $i$ with probability $V_{u,i}$ and move $j$ to child $i$.
 
 \item If $j$ is at a leaf $u$ with $C(u) < s$, then $j$ stays at $u$ and $C(u)$ increases by $1$.
 
If $j$ is at a leaf with $C(u) = s$, then the balls at $u$ are distributed among $u$ and its children as
follows. We select $s_{0} \leq s$ of the balls uniformly at random to stay at $u$. Among the remaining
$s + 1 - s_{0}$ balls, we uniformly at random distribute $s_{1}$ balls to each of the $b$ children of $u$.
Each of the remaining $s + 1 - s_{0} - bs_{1}$ balls is placed at a child vertex chosen independently at
random according to the split vector assigned to $u$. This splitting process is repeated for any
child which receives more than $s$ balls.
\end{enumerate}

We stop once all $n$ balls have been placed in $\mathbb{T}$ and we obtain $T_{n}^{{\rm sp}}$ by deleting all vertices $u \in \mathbb{T}$ such that the sub-tree rooted at $u$ contains no balls. Note that an internal vertex of $T_{n}^{{\rm sp}}$ contains exactly $s_{0}$ balls, while a leaf contains a random amount in $\{1, . . . , s\}$. Note also that in general the number $N^{(n)}$ of vertices of $T_{n}^{{\rm sp}}$ is a random variable while the number of balls $n$ is deterministic. 

Depending on the choice of the parameters $b, s, s_{0}, s_{1}$ and the distribution of $\mathcal{V}$, several important data structures may be modeled. For instance, binary search trees correspond to $b=2$, $s = s_{0} = 1$, $s_{1} = 0$ and $\mathcal{V}$ distributed as $(U, 1-U)$, where $U$ is a uniform random variable on $[0,1]$ (in this case $N^{(n)}=n$). Some other relevant (and more complicated) examples of split trees are $m$-ary search trees, median-of-$(2k+1)$ trees, quad trees, simplex trees; see \cite{Luc1999}, \cite{Holm2012} and \cite{Bro2012} for details and more examples.

In the present work, we assume without loss of generality that the components of the split vector $\mathcal{V}$ are identically distributed (even exchangeable); this can be done by using random permutations as explained in  \cite{Luc1999}. In particular, $\mathbb{E}[V_{1}] = 1/b$. We frequently use the following notation. Set
\begin{eqnarray} \label{eq56}
\mu \coloneqq b \mathbb{E}[-V_{1} \ln V_{1}]. 
\end{eqnarray}

\noindent Note that $\mu \in (0, \ln b)$ whenever $b \geq 2$. This quantity was first introduced by Devroye \cite{Luc1999} to study the depth of the last inserted ball of $T_{n}^{{\rm sp}}$ as the number of balls increases.

In the study of split trees, the following condition is often assumed:

\theoremstyle{Condition} 
\begin{Condition} \label{Cond1}
Assume that $\mathbb{P}(V_{1} = 1) = \mathbb{P}(V_{1} = 0) = 0$ and that $-\ln V_{1}$ is non-lattice, that is, there is no $a \in \mathbb{R}$ such that $-\ln V_{1} \in a \mathbb{Z}$ almost surely.
\end{Condition}

We sometimes also consider the following condition: 

\theoremstyle{Condition} 
\begin{Condition} \label{Cond2}
Suppose that, for some $\alpha > 0$ and $\varepsilon > 0$, 
\begin{eqnarray*}
\mathbb{E}[N^{(n)}] = \alpha n + O\left(\frac{n}{\ln^{1+\varepsilon}n} \right).
\end{eqnarray*}  
\end{Condition}

\noindent Recall that for two sequences of real numbers $(A_{n})_{n \geq 1}$ and $(B_{n})_{n \geq 1}$ such that $B_{n} >0$, one writes $A_{n} = O(B_{n})$ if $\sup_{n \geq 1} |A_{n}|/B_{n} < \infty$. Condition \ref{Cond2} first appears in \cite[Equation (52)]{Bro2012} for the study of the total path length of split trees. 

Holmgren \cite[Theorem 1.1]{Holm2012} showed that if $\ln V_{1}$ is non-lattice then there exists a constant $\alpha >0$ such that $\mathbb{E}[N^{(n)}] = \alpha n + o(n)$ and furthermore $Var(N^{(n)}) = o(n^{2})$. However, for technical reasons, the proof of Theorem \ref{Teo2} below requires the extra control on $\mathbb{E}[N^{(n)}]$ given in Condition \ref{Cond2}. We do not know whether Condition \ref{Cond2} is really necessary, and probably Theorem \ref{Teo2} still holds without such condition. We leave this as an open problem.

On the other hand, Condition \ref{Cond2} is satisfied in many interesting cases. For instance, it holds for $m$-ary search trees \cite{Mah1989}. Moreover, Flajolet et al.\ \cite{Flajolet2010} showed that for most tries (as long as $\ln V_{1}$ is non-lattice) Condition \ref{Cond2} holds. However, there are some special cases of random split trees that do not satisfy Condition \ref{Cond2}. For instance, tries (where $s=1$ and $s_{0} = 0$) with a fixed split vector $(1/b, \dots, 1/b)$, in which case $\ln V_{1}$ is lattice.

\subsection{Main results} \label{sec2}

In this section, we present the main results of this work. We consider Bernoulli bond-percolation with supercritical parameter $p_{n}$ satisfying (\ref{eq1}) on $T_{n}^{\rm sp}$. We denote by $C_{0}^{(n)}$ (resp.\ $\hat{C}_{0}^{(n)}$) the number of balls (resp.\ the number of vertices) of the cluster that contains the root and by $C_{1}^{(n)} \geq C_{2}^{(n)} \geq \cdots$ (resp.\ $\hat{C}_{1}^{(n)}\geq \hat{C}_{2}^{(n)} \geq \cdots$) the sequence of the number of balls (resp.\ the number of vertices) of the remaining clusters ranked in decreasing order. 



The first result corresponds to the size being defined as the number of balls in the cluster. We write $\overset{d}{\longrightarrow}$ to denote convergence in distribution.

\begin{theorem} \label{Teo1}
Let $T_{n}^{\rm sp}$ be a split tree that satisfies Condition \ref{Cond1} and suppose that $p_{n}$ fulfills (\ref{eq1}). Then, 
\begin{eqnarray*}
n^{-1} C_{0}^{(n)} \overset{\mathbb{P}}{\longrightarrow} e^{-c/\mu}, \hspace*{5mm} \text{as} \hspace*{2mm} n \rightarrow \infty,
\end{eqnarray*}

\noindent where the constants $c$ and $\mu$ are defined in (\ref{eq1}) and (\ref{eq56}), respectively. Furthermore, for every fixed $i \in \mathbb{N}$, 
\begin{eqnarray*}
\left( \frac{\ln n}{n} C_{1}^{(n)}, \dots, \frac{\ln n}{n} C_{i}^{(n)} \right) \overset{d}{\longrightarrow} ({\rm x}_{1}, \dots, {\rm x}_{i}), \hspace*{5mm} as \hspace*{2mm} n \rightarrow \infty,
\end{eqnarray*}

\noindent where ${\rm x}_{1} > {\rm x}_{2} > \cdots$ are the atoms of a Poisson process on $(0, \infty)$ with intensity $c\mu^{-1}e^{-c/\mu}x^{-2} {\rm d}x$.
\end{theorem}

Note that $1/ {\rm x}_{1}, 1/{\rm x}_{2}-1/{\rm x}_{1}, 1/{\rm x}_{3}-1/{\rm x}_{2}, \dots$ are i.i.d.\ exponential random variables with parameter $c\mu^{-1}e^{-c/\mu}$.

The second result corresponds to the size being defined as the number of vertices in the cluster.

\begin{theorem} \label{Teo2}
Let $T_{n}^{\rm sp}$ be a split tree that satisfies Conditions \ref{Cond1}-\ref{Cond2} and suppose that $p_{n}$ fulfills (\ref{eq1}). Then, 
\begin{eqnarray*}
n^{-1} \hat{C}_{0}^{(n)} \overset{\mathbb{P}}{\longrightarrow} \alpha e^{-c/\mu}, \hspace*{5mm} \text{as} \hspace*{2mm} n \rightarrow \infty,
\end{eqnarray*}

\noindent where the constants $c$, $\mu$ and $\alpha$ are defined in (\ref{eq1}), (\ref{eq56}) and Condition \ref{Cond2}, respectively. Furthermore, for every fixed $i \in \mathbb{N}$, 
\begin{eqnarray*}
\left( \frac{\ln n}{n} \hat{C}_{1}^{(n)}, \dots, \frac{\ln n}{n} \hat{C}_{i}^{(n)} \right) \overset{d}{\longrightarrow} ({\rm x}_{1}, \dots, {\rm x}_{i}), \hspace*{5mm} as \hspace*{2mm} n \rightarrow \infty,
\end{eqnarray*}

\noindent where ${\rm x}_{1} > {\rm x}_{2} > \cdots$ are the atoms of a Poisson process on $(0, \infty)$ with intensity $ c\alpha\mu^{-1}e^{-c/\mu}x^{-2} {\rm d}x$.
\end{theorem}


Note the similarity with the results for uniform random recursive trees in \cite{Be3} and preferential attachment trees in \cite{Uribe2015}. More precisely, the size of the second largest cluster, and more generally, the size of the $i$-th largest cluster (for $i \geq 2$) in the supercritical regime is of order $n/\ln n$ as in \cite{Be3} and \cite{Uribe2015}. Moreover, their sizes are described by the atoms of a Poisson process on $(0, \infty)$ whose intensity measure only differ by a constant factor. For example, for uniform random recursive trees \cite{Be3} the intensity is $ce^{-c}x^{-2} {\rm d}x$.

Recall that Condition \ref{Cond1} requires $- \ln V_{1}$ to be non-lattice. The next result  shows that Theorem \ref{Teo1} can essentially be extended to the  lattice case. Write $y = \lfloor y \rfloor + \{ y \}$ for the decomposition of a real number $y$ as the sum of its integer and fractional parts. For every $\varrho \in [0,1)$ and $a >0$, define the finite measure $\Xi_{\varrho}$ on $(0,\infty)$ by letting
\begin{eqnarray*}
\Xi_{\varrho}([x,\infty)) \coloneqq \frac{a}{1-e^{-a}} e^{a \lfloor \varrho -a^{-1}\ln x  \rfloor -a\varrho},  \hspace*{5mm} x >0.
\end{eqnarray*}

\begin{theorem} \label{Teo4}
Let $T_{n}^{\rm sp}$ be a split tree such that $\mathbb{P}(V_{1} = 1) = \mathbb{P}(V_{1} = 0) = 0$ and that $-\ln V_{1}$ is lattice with span $a>0$, that is, $- \ln V_{1} \in a \mathbb{Z}$ almost surely. Suppose also that $p_{n}$ fulfills (\ref{eq1}). Then, 
\begin{eqnarray*}
n^{-1} C_{0}^{(n)} \overset{\mathbb{P}}{\longrightarrow} e^{-c/\mu}, \hspace*{5mm} \text{as} \hspace*{2mm} n \rightarrow \infty,
\end{eqnarray*}

\noindent where the constants $c$ and $\mu$ are defined in (\ref{eq1}) and (\ref{eq56}), respectively. 

Furthermore, suppose that $n \rightarrow \infty$ such that $\{a^{-1} \ln \ln n\} \rightarrow \varrho \in [0,1)$. Then, for every fixed $i \in \mathbb{N}$, 
\begin{eqnarray*}
\left( \frac{\ln n}{n} C_{1}^{(n)}, \dots, \frac{\ln n}{n} C_{i}^{(n)} \right) \overset{d}{\longrightarrow} ({\rm x}_{1}, \dots, {\rm x}_{i}), \hspace*{5mm} as \hspace*{2mm} n \rightarrow \infty,
\end{eqnarray*}

\noindent where ${\rm x}_{1} \geq {\rm x}_{2} \geq \cdots$ are the atoms of a Poisson process on $(0, \infty)$ with intensity $c\mu^{-1}e^{-c/\mu}\Xi_{\varrho}({\rm d} x)$.
\end{theorem}

As we mentioned in the introduction, we shall follow a different route to that used in \cite{Be3} and \cite{Uribe2015}. Our approach is based on a remark made in \cite[Section 3]{Be1} about the behavior of the second largest cluster created by performing (supercritical) Bernoulli bond-percolation on complete regular trees. More precisely, consider a rooted complete regular $d$-ary tree $T_{h}^{\rm d}$ of height $h \in \mathbb{N}$, where $d \geq 2$ is some integer (i.e., each vertex has exactly out-degree $d$). There are $d^{k}$ vertices at distance $k = 0, 1, \dots, h$ from the root and a total of $(d^{h+1}-1)/(d-1)$ vertices. Perform Bernoulli bond-percolation with parameter
\begin{eqnarray}
q_{h} = 1-ch^{-1} + o(h^{-1}), \label{per}
\end{eqnarray}

\noindent where $c > 0$ is some fixed parameter. It has been shown in \cite[Section 3]{Be1}  that this choice of the percolation parameter corresponds precisely to the supercritical regime, that is, the root cluster is the unique giant component. Because the subtree rooted at a vertex at height $i \leq h$ is again a complete regular $d$-ary tree with height $h - i$, \cite[Corollary 1]{Be1} shows that the size (number of vertices) $G_{h}^{1}$ of the largest cluster which does not contain the root is close to
\begin{eqnarray*}
e^{-c}d^{h-\tau_{1}(h)+1}/(d-1),
\end{eqnarray*}

\noindent where $\tau_{1}(h)$ is the smallest height at which an edge has been removed. There are $d(d^{i}-1)/(d-1)$ edges with height at most $i$, so the distribution of $\tau_{1}(h)$ is given by
\begin{eqnarray*}
\mathbb{P}(\tau_{1}(h) > i) = q_{h}^{d(d^{i}-1)/(d-1)}, \hspace*{5mm} i = 1, \dots, h. 
\end{eqnarray*}

\noindent We use the notation $\log_{d} x = \ln x/\ln d$ for the logarithm with base $d$ of $x>0$.
It follows that in the regime (\ref{per}) and as soon as one assumes $\{ \log_{d}h \} \rightarrow \rho \in [0,1)$, as $h \rightarrow \infty$, that 
 $\tau_{1}(h) - \log_{d}h$ converges in distribution, and therefore, $hd^{-h} G_{h}^{1}$ also converges in distribution. 

Our strategy is then to adapt and improve the above argument to study the sizes of the next largest clusters in a random split tree with $n$ balls. We also show that this approach can be used to obtain a similar results
for supercritical percolation on complete $d$-regular trees of height $h \in \mathbb{N}$. More precisely, write $G_{0}^{(h)}$ for the number of vertices of the cluster that contains the root and  $G_{1}^{(h)} \geq G_{2}^{(h)} \geq \cdots$ for the sequence of the number vertices of the remaining clusters ranked in decreasing order.
We introduce for every $\rho \in [0,1)$ the measure $\Lambda_{\rho}$ on $(0,\infty)$ by letting
\begin{eqnarray*}
\Lambda_{\rho}([x,\infty)) \coloneqq d^{-\rho +\lfloor \rho -\log_{d}x \rfloor +1}/(d-1), \hspace*{5mm} x >0. 
\end{eqnarray*}

\begin{theorem} \label{Teo3}
Let $T_{h}^{\rm d}$ be a complete regular $d$-ary tree of height $h \in \mathbb{N}$ such that $\{ \log_{d}h \} \rightarrow \rho \in [0,1)$, as $h \rightarrow \infty$, and suppose that $q_{h}$ fulfills (\ref{per}). Then,
\begin{eqnarray*} 
 d^{-h} G_{0}^{(h)} \overset{\mathbb{P}}{\longrightarrow} de^{- c }/(d-1), \hspace*{5mm}  \text{as} \hspace*{2mm} h \rightarrow \infty, 
\end{eqnarray*}

\noindent where the constant $c$ is defined in (\ref{per}). Furthermore, for every fixed $i \in \mathbb{N}$, 
\begin{eqnarray*}
( hd^{-h} G_{1}^{(h)}, \dots, hd^{-h} G_{i}^{(h)} ) \overset{d}{\longrightarrow}  ({\rm x}_{1}, \dots, {\rm x}_{i}), \hspace*{5mm}  \text{as} \hspace*{2mm} h \rightarrow \infty,
\end{eqnarray*}

\noindent where ${\rm x}_{1} \geq {\rm x}_{2} \geq \cdots$ are the atoms of a Poisson process on $(0, \infty)$ with intensity $c\frac{d}{d-1}e^{-c} \Lambda_{\rho}({\rm d} x)$. 
\end{theorem}

The plan for the rest of this paper is as follows. Section \ref{sec3m} is devoted to the proof of Theorem \ref{Teo1}. In Section \ref{sec4}, we prove Theorem \ref{Teo2}. We show in Section \ref{prooflema} a law of large numbers (Proposition \ref{lemma1}) for the number of sub-trees in $T_{n}^{\rm sp}$ with cardinality (number of balls) at least $n/\ln n$, which may be of independent interest. (This is a technical ingredient used in the proof of Theorem \ref{Teo1}.) Finally, in Section \ref{sec5} and Section \ref{LatticeCase}, we show that an easy adaptation of the arguments used in the proof of Theorem \ref{Teo1} allows us to prove Theorem \ref{Teo4} and Theorem \ref{Teo3}, respectively

For the rest of the work, we remove the superscript $(n)$ (resp.\ $(h)$) from our notation $C_{i}^{(n)}$ and $\hat{C}_{i}^{(n)}$ (resp.\ $G_{i}^{(h)}$), and instead, we only write $C_{i}$ and $\hat{C}_{i}$ (resp.\ $G_{i}$) for simplicity. 

\section{Proof of Theorem \ref{Teo1}} \label{sec3m}

We split the proof of Theorem \ref{Teo1} in two parts. We start by studying the sizes of percolated sub-trees that are close to the root. One could refer to these percolated sub-trees as the early clusters since their distance to the root is the smallest. Then we show that the largest percolation clusters can be found amongst those (early) percolated sub-trees. 

\subsection{Sizes of early clusters} \label{sec3}

For $i \in \mathbb{N}$, let ${\bf e}_{i,n}$ be the edge with the $i$-th smallest height (we break ties by ordering the edges from left
to right, however, the order is not relevant in the proofs) that has been removed and ${\bf v}_{i,n}$ the endpoint (vertex) of ${\bf e}_{i,n}$ that is the furthest away from the root of $T_{n}^{\rm sp}$. Let $T_{i,n}$ be the sub-tree of $T_{n}^{\rm sp}$ that is rooted at ${\bf v}_{i,n}$ and let $n_{i,n}$ be the number of balls stored in the sub-tree $T_{i,n}$. For $t \in [0,\infty)$, we write
\begin{eqnarray*}
N_{n}(0) \coloneqq  0 \hspace*{3mm} \text{and} \hspace*{3mm} N_{n}(t) \coloneqq \sum_{i \geq 1} \mathds{1}_{\left\{n_{i,n} \geq \frac{n}{t \ln n} \right\}} = \sum_{i \geq 1} \mathds{1}_{\left\{ (n/n_{i,n})\frac{1}{\ln n} \leq t \right\}}
\end{eqnarray*}

\noindent for the number of sub-trees $T_{i,n}$ that store at least $\lfloor n/(t\ln n) \rfloor$ balls. 

\begin{theorem} \label{Pro1}
Suppose that Condition \ref{Cond1} holds and that $p_{n}$ fulfills (\ref{eq1}). Then, the following convergence holds in the sense of weak convergence of finite dimensional distributions,
\begin{eqnarray*}
(N_{n}(t), t \geq 0) \overset{d}{\longrightarrow} (N(t), t \geq 0), \hspace*{5mm} \text{as} \hspace*{2mm} n \rightarrow \infty,
\end{eqnarray*}

\noindent where $(N(t), t \geq 0)$ is a Poisson process with intensity $c \mu^{-1}$.
\end{theorem} 

We expect that the convergence in Theorem \ref{Pro1} can be
improved in order to show convergence in distribution of the process $(N_{n}(t), t \geq 0)$ for the Skorohod topology on the space $\mathbb{D}([0, \infty), \mathbb{R})$ of right-continuous functions with left limits to a Poisson process with intensity $c \mu^{-1}$; see, for instance, \cite[Theorem 12.6, Chapter 3]{Billi1999}. We leave this as an open problem.

The proof of Theorem \ref{Pro1} uses the following result which provides a law of large numbers for the number of sub-trees in $T_{n}^{\rm sp}$ with cardinality at least  $n/(t \ln n)$. More precisely, for a vertex $v \in T_{n}^{\rm sp}$ that is not the root $\circ$, let $n_{v}$ denote the number of balls stored in the sub-tree of $T_{n}^{\rm sp}$ rooted at $v$. Define
\begin{eqnarray}  \label{MM}
M_{n}(t) \coloneqq \# \left\{ v \in T_{n}^{\rm sp}: v \neq \circ \hspace*{3mm} \text{and} \hspace*{3mm} n_{v} \geq \frac{n}{t\ln n} \right\}, \hspace*{5mm} \text{for} \hspace*{2mm} t \in [0, \infty). 
\end{eqnarray}

\begin{proposition} \label{lemma1}
Suppose that Condition \ref{Cond1} holds. Then, for every fixed $t \in [0, \infty)$, we have that
$ (\ln n)^{-1}M_{n}(t) \overset{\mathbb{P}}{\longrightarrow} \mu^{-1}t$, as $n \rightarrow \infty$.
\end{proposition}

The proof of Proposition \ref{lemma1} is rather technical and it is postponed to Section \ref{prooflema}. 

\begin{proof}[Proof of Theorem \ref{Pro1}]
For a vertex $v \in T_{n}^{\rm sp}$ that is not the root $\circ$, let ${\bf e}_{v}$ be the edge that connects $v$ with its parent. Define the event $E_{v} \coloneqq \{ \text{the edge} \, \, {\bf e}_{v} \, \, \text{has been removed after percolation} \}$ and write $\xi_{v} \coloneqq \mathds{1}_{E_{v}}$. So, $(\xi_{v})_{v \neq \circ}$ is a sequence of i.i.d. Bernoulli random variables with parameter $1-p_{n}$ (that is, the probability of removing an edge). Then, it is clear that
\begin{eqnarray*}
N_{n}(t) = \sum_{v \neq \circ} \mathds{1}_{\left\{n_{v} \geq \frac{n}{t \ln n} \right\}} \xi_{v}, \hspace*{5mm} t \in [0,\infty).
\end{eqnarray*}
Let $\mathcal{F}$ be the $\sigma$-algebra generated by $(n_{v})_{v \neq \circ}$. Note that the variables $(\xi_{v})_{v \neq \circ}$ are independent of $\mathcal{F}$. Conditioning on $\mathcal{F}$, we have that $(N_{n}(t), t \geq 0)$ has independent increments and $N_{n}(t) \stackrel{d}{=} {\rm Bin} \left( M_{n}(t), 1-p_{n}  \right)$, where ${\rm Bin}(m,q)$ denotes a binomial $(m,q)$ random variable. Moreover,\begin{eqnarray*}
N_{n}(t) - N_{n}(s) \stackrel{d}{=} {\rm Bin} \left( M_{n}(t) - M_{n}(s), 1-p_{n}  \right), 
\end{eqnarray*}

\noindent for $0 \leq s \leq t$. Since $p_{n}$ fulfills (\ref{eq1}), we have that $1- p_{n} \rightarrow 0$ as $n \rightarrow \infty$. Furthermore, Proposition \ref{lemma1} implies that 
$(1- p_{n})(M_{n}(t) - M_{n}(s)) \overset{\mathbb{P}}{\longrightarrow} c\mu^{-1}(t-s)$, as $n \rightarrow \infty$. Therefore, without conditioning on $\mathcal{F}$ (by the dominated convergence theorem),
\begin{eqnarray*}
N_{n}(t) - N_{n}(s) \overset{d}{\longrightarrow}  {\rm  Poisson}(c\mu^{-1} (t-s)),  \hspace*{5mm} \text{as} \hspace*{2mm} n \rightarrow \infty, 
\end{eqnarray*}
\noindent where ${\rm Poisson}(\lambda)$ denotes a Poisson random variable with mean $\lambda$. Therefore, our result follows by the law of rare events.
\end{proof}

\begin{corollary} \label{cor1}
Suppose that Condition \ref{Cond1} holds and that $p_{n}$ fulfills (\ref{eq1}). Then, for every fixed $i \in \mathbb{N}$, 
\begin{eqnarray*}
\left( \frac{\ln n}{n} n_{1,n}, \dots, \frac{\ln n}{n} n_{i,n} \right) \overset{d}{\longrightarrow} ({\rm x}_{1}, \dots, {\rm x}_{i}), \hspace*{5mm} \text{as} \hspace*{2mm} n \rightarrow \infty,
\end{eqnarray*}
\noindent where ${\rm x}_{1} > {\rm x}_{2} > \cdots$ are the atoms of a Poisson process on $(0, \infty)$ with intensity $c \mu^{-1}x^{-2}{\rm d} x$. 
\end{corollary}

\begin{proof}
Note that $(n/n_{1,n}) \frac{1}{\ln n} \leq (n/n_{2,n}) \frac{1}{\ln n} \leq \cdots$ are the atoms (or occurrence times) of the counting process $(N_{n}(t), t \geq 0)$ ranked in increasing order. Theorem \ref{Pro1} implies that for every fixed $i \in \mathbb{N}$, 
\begin{eqnarray*}
\left( \frac{n}{n_{1,n} \ln n}, \dots, \frac{n}{n_{i,n} \ln n} \right) \overset{d}{\longrightarrow} ({\rm y}_{1}, \dots, {\rm y}_{i}),  \hspace*{5mm} \text{as} \hspace*{2mm} n \rightarrow \infty,
\end{eqnarray*}
\noindent where ${\rm y}_{1} < {\rm y}_{2} < \cdots$ are the atoms of the Poisson process $(N(t), t \geq 0)$ of rate $c\mu^{-1}$. To see this, note that the atoms are determined by the values of the random variables $N_{n}(t)$ and $N(t)$: For instance, the event $(n/n_{1,n}) \frac{1}{\ln n} > s$ and $ (n/n_{2,n}) \frac{1}{\ln n} > t$, for $0 \leq s \leq t$, is the same as the event $N_{n}(s)<1$ and $N_{n}(t)<2$. Therefore, the continuous mapping theorem (\cite[Theorem 2.7, Chapter 1]{Billi1999}) implies that
\begin{eqnarray*}
\left( \frac{\ln n}{n} n_{1,n}, \dots, \frac{\ln n}{n} n_{i,n} \right) \overset{d}{\longrightarrow} (1/{\rm y}_{1}, \dots, 1/{\rm y}_{i}),\hspace*{5mm} \text{as} \hspace*{2mm} n \rightarrow \infty.
\end{eqnarray*} 
\noindent Our claim then follows from basic properties of Poisson processes (\cite[Proposition 3.7, Chapter 3]{Res1987}).
\end{proof}

\subsection{Asymptotic sizes of the largest percolation clusters}

Recall that, for $i \in \mathbb{N}$, we let ${\bf e}_{i,n}$ be the edge with the $i$-th smallest height that has been removed and ${\bf v}_{i,n}$ the endpoint (vertex) of ${\bf e}_{i,n}$ that is the furthest away from the root of $T_{n}^{\rm sp}$. Recall also that $T_{i,n}$ denotes the sub-tree of $T_{n}^{\rm sp}$ that is rooted at ${\bf v}_{i,n}$ and that we write $n_{i,n}$ for the number of balls stored in the sub-tree $T_{i,n}$. We denote by $\tilde{C}_{i}$ the size (number of balls) of the root-cluster of $T_{i,n}$ after performing percolation (where here of course root means ${\bf v}_{i,n}$). We also write $\tilde{C}_{i}^{\ast}$ for the size (number of balls) of the second largest cluster of $T_{i,n}$ that does not contain its root.

In the sequel, we shall use the following notation $A_{n} = B_{n}+ o_{\rm p}(f(n))$, where $A_{n}$ and $B_{n}$ are two sequences of real random variables and $f: \mathbb{N} \rightarrow (0, \infty)$ is a function, to indicate that $\lim_{n \rightarrow \infty} |A_{n} - B_{n}|/f(n) = 0$ in probability.  

\begin{proposition} \label{Pro2}
Suppose that Condition \ref{Cond1} holds and that $p_{n}$ fulfills (\ref{eq1}). For every fixed $i \in \mathbb{N}$, 
\begin{eqnarray*}
\tilde{C}_{i}^{\ast} = o_{\rm p}(n/ \ln n).
\end{eqnarray*}
\noindent Furthermore, we have that
\begin{eqnarray*}
\left( \frac{\tilde{C}_{1}}{n_{1,n}} , \dots, \frac{\tilde{C}_{i}}{n_{i,n}}  \right) \overset{\mathbb{P}}{\longrightarrow} (e^{-c/\mu}, \dots, e^{-c/\mu}),\hspace*{5mm} \text{as} \hspace*{2mm} n \rightarrow \infty.
\end{eqnarray*}
\end{proposition}

\begin{proof}
It suffices to show our claim for every fixed $j \in \{1, \dots, i\}$ since convergence in probability to a constant implies the joint convergence for every fixed $i \in \mathbb{N}$. Given $n_{j,n}$, we see that $T_{j,n}$ is a split tree with $n_{j,n}$ balls. Then supercritical Bernoulli bond-percolation in $T_{j,n}$ corresponds to percolation parameters satisfying
\begin{eqnarray*}
1-p_{n_{j,n}} = c/\ln n_{j,n} + o\left( 1/\ln n_{j,n}\right),
\end{eqnarray*}
\noindent where $c >0$ is fixed. Corollary \ref{cor1} implies that $(\ln n_{j,n})/\ln n \overset{\mathbb{P}}{\longrightarrow}1$, as $n \rightarrow \infty$. Hence 
\begin{eqnarray*}
1-p_{n_{j,n}} = 1-p_{n} + o_{\rm p}\left( 1/\ln n \right).
\end{eqnarray*}
\noindent Therefore, an application of \cite[Lemma 2]{Ber2019} shows that $\tilde{C}_{j}/n_{j,n} \overset{\mathbb{P}}{\longrightarrow} e^{-c/\mu}$, as $n \rightarrow \infty$, which proves the second assertion. Moreover, \cite[Lemma 2]{Ber2019} also shows that $\tilde{C}_{j}^{\ast}/n_{j,n} \overset{\mathbb{P}}{\longrightarrow} 0$, as $n \rightarrow \infty$, and by Corollary \ref{cor1}, we conclude that $\tilde{C}_{j}^{\ast} = o_{\rm p}(n/\ln n)$. This completes the proof. 
\end{proof}

\begin{corollary} \label{cor2}
Suppose that Condition \ref{Cond1} holds and that $p_{n}$ fulfills (\ref{eq1}). Then, for every fixed $i \in \mathbb{N}$, 
\begin{eqnarray*}
\left( \frac{\ln n}{n} \tilde{C}_{1}, \dots, \frac{\ln n}{ n} \tilde{C}_{i}  \right) \xrightarrow[]{d} ({\rm x}_{1}, \dots, {\rm x}_{i}) ,\hspace*{5mm} \text{as} \hspace*{2mm} n \rightarrow \infty,
\end{eqnarray*}
\noindent where ${\rm x}_{1} >{\rm x}_{2} > \cdots$ are the atoms of a Poisson process on $(0, \infty)$ with intensity $c\mu^{-1}e^{-c/\mu}x^{-2} {\rm d}x$. 
\end{corollary}

\begin{proof}
For every $i \in \mathbb{N}$ fixed, Corollary \ref{cor1} and Proposition \ref{Pro2} together with \cite[Theorem 3.9]{Billi1999} imply that 
\begin{eqnarray*}
\left( \frac{\ln n}{ n} n_{1,n}, \dots, \frac{\ln n}{ n} n_{i,n}, \frac{\tilde{C}_{1}}{n_{1,n}} , \dots, \frac{\tilde{C}_{i}}{n_{i,n}}  \right)\overset{d}{\longrightarrow} \left(\rm{y}_{1}, \dots, \rm{y}_{i}, e^{-c/\mu}, \dots, e^{-c/\mu} \right), \hspace*{5mm} \text{as} \hspace*{2mm} n \rightarrow \infty,
\end{eqnarray*} 
\noindent where ${\rm y}_{1} >{\rm y}_{2} > \cdots$ are the atoms of a Poisson process on $(0, \infty)$ with intensity $c \mu^{-1}y^{-2}{\rm d} y$. Define the function $H: \mathbb{R}^{2i} \rightarrow \mathbb{R}^{i}$ by $H(x_{1}, \dots x_{2i}) = (x_{1}x_{i+1} , \dots, x_{i}x_{2i})$ for all $(x_{1}, \dots, x_{2i}) \in \mathbb{R}^{2i}$. 
Observe that $H$ is continuous and that
\begin{eqnarray*}
\left( \frac{\ln n}{ n} \tilde{C}_{1}, \dots, \frac{\ln n}{n} \tilde{C}_{i} \right) = H \left( \frac{\ln n}{ n} n_{1,n}, \dots, \frac{\ln n}{ n} n_{i,n}, \frac{\tilde{C}_{1}}{n_{1,n}} , \dots, \frac{\tilde{C}_{i}}{n_{i,n}}  \right).
\end{eqnarray*}
\noindent Therefore, by the continuous mapping theorem (\cite[Theorem 2.7, Chapter 1]{Billi1999}),
\begin{eqnarray*}
\left( \frac{\ln n}{n} \tilde{C}_{1}, \dots, \frac{\ln n}{ n} \tilde{C}_{i} \right) \overset{d}{\longrightarrow} ({\rm y}_{1}e^{-c/\mu}, \dots, {\rm y}_{i}e^{-c/\mu}), \hspace*{5mm} \text{as} \hspace*{2mm} n \rightarrow \infty,
\end{eqnarray*}
\noindent and our claim follows from basic distributional properties of Poisson processes.
\end{proof}

The last ingredient in the proof of Theorem \ref{Teo1} consists of verifying that for every fixed $i \in \mathbb{N}$, one can choose $\ell \in \mathbb{N}$ large enough such that with probability tending to $1$, as $n \rightarrow \infty$, the $i$-th largest percolation cluster of $T_{n}^{\rm sp}$ can be found amongst the root-clusters of the percolated tree-components $T_{1,n}, \dots, T_{\ell,n}$. Rigorously, denote by
$\tilde{C}_{1, \ell} \geq \tilde{C}_{2, \ell} \geq \cdots \geq \tilde{C}_{\ell, \ell}$ the rearrangement in decreasing order of the $\tilde{C}_{i}$ for $i =1, \dots, \ell$. Recall that $C_{i}$ stands for the size (number of balls) of the $i$-th largest cluster (that does not contain the root).

\begin{lemma} \label{lemma2}
Suppose that Condition \ref{Cond1} holds and that $p_{n}$ fulfills (\ref{eq1}). Then for each fixed $i \in \mathbb{N}$, 
\begin{eqnarray*}
\lim_{\ell \rightarrow \infty} \liminf_{n \rightarrow \infty} \mathbb{P} \left(\tilde{C}_{k, \ell}= C_{k} \hspace*{3mm} \text{for every} \hspace*{2mm} k =1, \dots, i    \right) = 1.
\end{eqnarray*}
\end{lemma}

\begin{proof}
A Poisson process on $(0,\infty)$ with intensity $c\mu^{-1}e^{-c/\mu}x^{-2} {\rm d}x$ has infinitely many atoms. Moreover, in the notation of Corollary \ref{cor2}, a.s.\ $\min \{{\rm  x}_{1}, \dots,  {\rm  x}_{i} \} >0$. Note that $C_{i}$ cannot be smaller than $\min \{ \tilde{C}_{1}, \dots, \tilde{C}_{i} \}$. Then our claim follows from Corollary \ref{cor2} and along the lines of the proof of \cite[Lemma 6]{Be3}. 
\end{proof}

We can now finish the proof of Theorem \ref{Teo1}. 

\begin{proof}[Proof of Theorem \ref{Teo1}]
We have already proven the first claim in \cite[Lemma 2]{Ber2019}. We only prove the second claim. For every fixed $i \in \mathbb{N}$, consider a continuous function $f:[0,\infty)^{i} \rightarrow [0,1]$ and fix $\varepsilon >0$. According to Lemma \ref{lemma2}, we may choose $\ell \in \mathbb{N}$ sufficiently large so that there exists $n_{\varepsilon} \in \mathbb{N}$ such that
\begin{eqnarray*}
\mathbb{E}\left[ f \left( \frac{\ln n}{n} C_{1}, \dots,  \frac{\ln n}{n} C_{i} \right) \right] \leq \mathbb{E}\left[ f \left( \frac{\ln n}{n} \tilde{C}_{1,\ell}, \dots,  \frac{\ln n}{n}\tilde{C}_{i,\ell} \right) \right]  + \varepsilon
\end{eqnarray*}
\noindent holds for all $n \geq n_{\varepsilon}$. We then deduce from Corollary \ref{cor2} and the previous bound that
\begin{eqnarray*}
\limsup_{n \rightarrow \infty} \mathbb{E}\left[ f \left( \frac{\ln n}{n} C_{1}, \dots,  \frac{\ln n}{n} C_{i} \right) \right] \leq \mathbb{E}\left[ f \left( {\rm x}_{1}, \dots, {\rm x}_{i} \right) \right]  + \varepsilon.
\end{eqnarray*}
\noindent Since $\varepsilon >0$ can be arbitrary small and $f$ can be replaced by $1-f$, this establishes Theorem \ref{Teo1}. 
\end{proof}

\section{Proof of Theorem \ref{Teo2}} \label{sec4}

In this section, we prove Theorem \ref{Teo2} along similar lines as in the proof of Theorem \ref{Teo1}. Indeed, we only need a version of Proposition \ref{Pro2} where we consider that cluster sizes are given by the number of vertices instead of the number of balls. 

For $i \in \mathbb{N}$, recall the definition given in Section \ref{sec3} of  the sub-trees $T_{i,n}$ rooted at the vertex ${\bf v}_{i,n}$. Recall also that $n_{i,n}$ denotes the number of balls stored at $T_{i,n}$. We denote by $\bar{C}_{i}$ the size (number of vertices) of the root-cluster of $T_{i,n}$ after performing percolation (where here of course root means ${\bf v}_{i,n}$). We also write $\bar{C}_{i}^{\ast}$ for the size (number of vertices) of the second largest cluster of $T_{i,n}$ that does not contain its root.

\begin{proposition} \label{Pro3}
Suppose that Conditions \ref{Cond1}  and \ref{Cond2}  hold and that $p_{n}$ fulfills (\ref{eq1}). For every fixed $i \in \mathbb{N}$, 
\begin{eqnarray*}
\bar{C}_{i}^{\ast} = o_{\rm p}(n/ \ln n).
\end{eqnarray*}
\noindent Recall the constant $\alpha$ defined in Condition \ref{Cond2}. Then, we also have that
\begin{eqnarray*}
\left( \frac{\bar{C}_{1}}{n_{1,n}} , \dots, \frac{\bar{C}_{i}}{n_{i,n}}  \right) \overset{\mathbb{P}}{\longrightarrow} (\alpha e^{-c/\mu}, \dots, \alpha e^{-c/\mu}),\hspace*{5mm} \text{as} \hspace*{2mm} n \rightarrow \infty.
\end{eqnarray*}
\end{proposition}

\begin{proof}
The same argument as in the proof of Proposition \ref{Pro2} shows that, under Conditions \ref{Cond1}-\ref{Cond2}, the supercritical percolation regime in $T_{j,n}$ corresponds to parameters fulfilling (\ref{eq1}). That is, \cite[Lemma 3]{Ber2019} implies that $\bar{C}_{j}/n_{j,n} \overset{\mathbb{P}}{\longrightarrow} \alpha e^{-c/\mu}$, as $n \rightarrow \infty$ which proves the second assertion. Moreover, \cite[Lemma 3]{Ber2019} also shows that $\bar{C}_{j}^{\ast}/n_{j,n} \overset{\mathbb{P}}{\longrightarrow} 0$, as $n \rightarrow \infty$, and by Corollary \ref{cor1}, $\bar{C}_{j}^{\ast} = o_{\rm p}(n/\ln n)$. 
\end{proof}

We can now prove Theorem \ref{Teo2}. We only provide enough details to convince the reader that everything can be carried out as in the proof of Theorem \ref{Teo1}.

\begin{proof}[Proof of Theorem \ref{Teo2}]
The first claim has been proved in \cite[Lemma 3]{Ber2019}, and thus, we only prove the second one. Following exactly the same argument as in the proof of Corollary \ref{cor2}, we deduce from Corollary \ref{cor1} and Proposition \ref{Pro3} that for every fixed $i \in \mathbb{N}$,
\begin{eqnarray*}
\left( \frac{\ln n}{n} \bar{C}_{1}, \dots, \frac{\ln n}{ n} \bar{C}_{i}  \right) \overset{d}{\longrightarrow} ({\rm x}_{1}, \dots, {\rm x}_{i}),\hspace*{5mm} \text{as} \hspace*{2mm} n \rightarrow \infty,
\end{eqnarray*}
\noindent where ${\rm x}_{1} >{\rm x}_{2} > \cdots$ are the atoms of a Poisson process on $(0, \infty)$ with intensity $c\alpha \mu^{-1}e^{-c/\mu}x^{-2} {\rm d}x$. 
For $\ell \in  \mathbb{N}$, denote by $ \bar{C}_{1, \ell} \geq \bar{C}_{2, \ell} \geq \cdots \geq \bar{C}_{\ell, \ell} $ the rearrangement in decreasing order of the $\bar{C}_{i}$ for $i =1, \dots, \ell$. Following the proof of Lemma \ref{lemma2}, one can show that for every fixed $i \in \mathbb{N}$,
\begin{eqnarray*}
\lim_{\ell \rightarrow \infty} \liminf_{n \rightarrow \infty} \mathbb{P} \left(\bar{C}_{k, \ell}= \hat{C}_{k} \hspace*{3mm} \text{for every} \hspace*{2mm} k =1, \dots, i    \right) = 1.
\end{eqnarray*}

Finally, by combining the previous two facts, the proof of Theorem \ref{Teo2} is completed in analogy to the proof of Theorem \ref{Teo1}.
\end{proof}

\section{Proof of Proposition \ref{lemma1}} \label{prooflema}

This section is devoted to the proof of Proposition \ref{lemma1}. We start by recalling some well-known properties of random split trees. For a vertex $v \in T_{n}^{\rm sp}$ that is not the root $\circ$, let $n_{v}$ denote the number of balls stored at the sub-tree of $T_{n}^{\rm sp}$ rooted at $v$. Let $d_{n}(v)$ denote the depth (or height) of the vertex $v$ in $T_{n}^{\rm sp}$. Let $(V_{v,k}: k=1, \dots, d_{n}(v))$ be the collection of i.i.d.\ random variables on $[0,1]$ given by the split vectors associated with the vertices in the unique path from $v$ to the root $\circ$ of $T_{n}^{\rm sp}$.  In particular, $V_{v,k} = V_{1}$ in distribution which implies that $\mathbb{E}[V_{v,k}] = \mathbb{E}[V_{1}] = 1/b$ and $\mathbb{E}[V_{v,k}^{2}] = \mathbb{E}[V_{1}^{2}] < 1/b$. Then let $\mathcal{L}_{v} \coloneqq \prod_{k=1}^{d_{n}(v)} V_{v,k}$. If $d_{n}(v) =i$, conditioning on the split vectors, it is well-known that $n_{v}$ is in the stochastic sense bounded by the following random variables
\begin{eqnarray} \label{eq4}
\text{Bin}\left(n, \mathcal{L}_{v}   \right) - si \leq n_{v} \leq
\text{Bin}\left(n, \mathcal{L}_{v}  \right) + s_{1}i;
\end{eqnarray}
\noindent this property has been used before in \cite{Luc1999} and \cite{Holm2012}.  By calculating the first and second moment of the Binomial distribution, it follows that
$E[n_{v}] \leq n  \mathbb{E}[\mathcal{L}_{v} ] + s_{1}i = n b^{-i} + s_{1}i$ and
\begin{eqnarray} \label{eq19}
E[n_{v}^{2}]  \leq  n^{2}   \mathbb{E}[\mathcal{L}_{v} ^{2}] + n\left(\mathbb{E}[\mathcal{L}_{v} ] -  \mathbb{E}[\mathcal{L}_{v} ^{2}] \right) + 2i s_{1} n \mathbb{E}[\mathcal{L}_{v} ] + s_{1}^{2}i^{2}. 
\end{eqnarray}

We start by proving some crucial lemmas that are used in the proof of Proposition \ref{lemma1}. For $t \in [0, \infty)$, recall the definition of $M_{n}(t)$ in (\ref{MM}). Recall also that we use the notation $\log_{b}x = \ln x/\ln b$ for the logarithm with base $b$ of $x > 0$. We then write $m_{n} =  \lfloor \beta \log_{b} \ln n \rfloor$ for some large constant $\beta > 0$. For $t \in [0, \infty)$, we define 
\begin{eqnarray*}
M_{n}^{(1)}(t) \coloneqq \# \left\{ v \in T_{n}^{\rm sp}: 1 \leq d_{n}(v) \leq m_{n} \hspace*{3mm} \text{and} \hspace*{3mm} n_{v} \geq \frac{n}{t\ln n} \right\}
\end{eqnarray*}
\noindent and
\begin{eqnarray*}
M_{n}^{(2)}(t) \coloneqq M_{n}(t) - M_{n}^{(1)}(t) = \# \left\{ v \in T_{n}^{\rm sp}: d_{n}(v) > m_{n} \hspace*{3mm} \text{and} \hspace*{3mm} n_{v} \geq \frac{n}{t\ln n} \right\}.
\end{eqnarray*}

\begin{lemma} \label{lemma3}
Suppose that Condition \ref{Cond1} holds. One can choose $\beta >0$ large enough such that for every fixed $t \in [0, \infty)$ we have that $ (\ln n)^{-2}\mathbb{E}[(M_{n}^{(2)}(t))^{2}] \rightarrow 0$ as $n \rightarrow \infty$.
\end{lemma}

\begin{proof}
For some constant $C>0$, we write $\tilde{m}_{n} = \lfloor C\ln n  \rfloor$ and define 
\begin{eqnarray*}
X_{n}(t) \coloneqq \# \left\{ v \in T_{n}^{\rm sp}: m_{n} < d_{n}(v) \leq \tilde{m}_{n} \hspace*{2mm} \text{and} \hspace*{2mm} n_{v} \geq \frac{n}{t\ln n}  \right\}
\end{eqnarray*}
\noindent and
\begin{eqnarray*}
X_{n}^{\rm c}(t) \coloneqq M_{n}^{(2)}(t) - X_{n}(t) = \# \left\{ v \in T_{n}^{\rm sp}: d_{n}(v) > \tilde{m}_{n} \hspace*{2mm} \text{and} \hspace*{2mm} n_{v} \geq \frac{n}{t\ln n}  \right\}.
\end{eqnarray*}
\noindent Then 
\begin{eqnarray*}
\mathbb{E}[(M_{n}^{(2)}(t))^{2}] = \mathbb{E}[(X_{n}(t))^{2}] + 2\mathbb{E}[X_{n}(t) X_{n}^{\rm c}(t)] + \mathbb{E}[(X_{n}^{\rm c}(t))^{2}].
\end{eqnarray*}

\noindent Note that $X_{n}^{\rm c}(t) \leq \# \left\{ v \in T_{n}^{\rm sp}:  d_{n}(v) > \tilde{m}_{n} \right\}$. Then, \cite[Remark 3.4]{Holm2012} allows us to choose $C>0$ such that $\mathbb{E}[\# \left\{ v \in T_{n}^{\rm sp}:  d_{n}(v) > \tilde{m}_{n} \right\}] = O(n^{-1})$ and thus, we see that $(\ln n)^{-2}\mathbb{E}[ (X_{n}^{\rm c}(t))^{2}] \rightarrow 0$, as $n \rightarrow \infty$. Hence we only need to check that $(\ln n)^{-2}\mathbb{E}[(X_{n}(t))^{2}] \rightarrow 0$, as $n \rightarrow \infty$, since Cauchy–Schwarz inequality would imply that 
\begin{eqnarray*}
(\ln n)^{-2}\mathbb{E}[X_{n}(t) X_{n}^{\rm c}(t)] \leq \left( (\ln n)^{-2} \mathbb{E}[(X_{n}(t))^{2}] \right)^{1/2} \left( (\ln n)^{-2} \mathbb{E}[(X_{n}^{\rm c}(t))^{2}] \right)^{-1/2} \rightarrow 0, \hspace*{4mm} \text{as} \hspace*{2mm} n \rightarrow \infty. 
\end{eqnarray*}

For $v \in T_{n}^{\text{sp}}$ such that $v \neq \circ$ and $i \in \mathbb{N}$, we define the events
\begin{eqnarray*}
A_{v} \coloneqq \left \{ n_{v} \geq \frac{n}{t\ln n} \right \} \hspace*{3mm} \text{and} \hspace*{3mm} B_{i} \coloneqq \left \{ \exists \hspace{2mm} v \in T_{n}^{\text{sp}} \hspace*{2mm} \text{with} \hspace*{2mm} d_{n}(v) = i \hspace*{2mm} \text{and such that} \hspace*{2mm}  n_{v} \geq \frac{n}{t\ln n} \right \}. 
\end{eqnarray*}
\noindent Note that the number of sub-trees rooted at vertices with depth $d$ for any $d >0$, and which, store at least $n/(t \ln n)$ balls have to be less than $\varepsilon t \ln n$, for any $\varepsilon >1$. Otherwise, one would contradict that the total number of balls (or keys) in $\mathbb{T}_{n}^{\rm sp}$ is equal to $n$. Then,
\begin{eqnarray} \label{eq7}
\mathds{1}_{B_{i}} \leq \sum_{v \neq \circ} \mathds{1}_{A_{v}} \mathds{1}_{\{ d_{n}(v) = i \}} \leq \varepsilon (t\ln n) \mathds{1}_{B_{i}},
\end{eqnarray}
\noindent for $\varepsilon>1$. Let $W_{i}$, $i \geq 1$ be i.i.d.\ copies of $V_{1}$, and define $L_{k} = \prod_{i=1}^{k} W_{i}$ for $k \in \mathbb{N}$. Hence, the inequality (\ref{eq4}) implies that for $n$ large enough and $m_{n} < i \leq \tilde{m}_{n}$,
\begin{eqnarray*}
\mathbb{P}(B_{i}) \leq \mathbb{E} \left[  \sum_{v \neq \circ} \mathds{1}_{A_{v}} \mathds{1}_{\{ d_{n}(v) = i \}} \right] \leq b^{i} \mathbb{P}\left({\rm Bin}(n, L_{i}) + s_{1}i \geq \frac{n}{t\ln n} \right) \leq b^{i} \mathbb{P}\left({\rm Bin}(n, L_{i}) \geq \frac{n}{2t \ln n} \right).
\end{eqnarray*}

\noindent The second moment of a ${\rm Bin}(m, q)$ is $m^{2}q^{2}+mq-mq^{2}$. Then, Markov's inequality shows that 
\begin{eqnarray} \label{eq8}
\mathbb{P}(B_{i}) \leq (2 t \ln n)^{2}b^{i} \left( \mathbb{E}[L_{i}^{2}] + n^{-1}(\mathbb{E}[L_{i}] - \mathbb{E}[L_{i}^{2}]) \right) \leq (2 t \ln n)^{2} \left((b \mathbb{E}[V_{1}^{2}])^{i} + n^{-1} \right)
\end{eqnarray}
\noindent since $\mathbb{E}[V_{1}^{2}] < \mathbb{E}[V_{1}] = 1/b < 1$. The inequality (\ref{eq7}) implies that 
\begin{eqnarray*}
X_{n}(t) \leq \sum_{m_{n} < i \leq \tilde{m}_{n}} \sum_{v \neq \circ} \mathds{1}_{A_{v}} \mathds{1}_{\{ d_{n}(v) = i \}}  \leq \varepsilon (t\ln n)  \sum_{m_{n} < i \leq \tilde{m}_{n}} \mathds{1}_{B_{i}}.
\end{eqnarray*}
\noindent Then the Cauchy–Schwarz inequality shows that
\begin{eqnarray*}
\mathbb{E}[(X_{n}(t))^{2}] \leq \varepsilon^{2} (t\ln n)^{2} \sum_{m_{n} < i,j \leq \tilde{m}_{n}}  \mathbb{E}[ \mathds{1}_{B_{i}} \mathds{1}_{B_{j}} ] \leq \varepsilon^{2} (t\ln n)^{2} \sum_{m_{n} < i,j \leq \tilde{m}_{n}} (\mathbb{P}(B_{i})  \mathbb{P}( B_{j}))^{1/2}.
\end{eqnarray*}

\noindent We conclude from (\ref{eq8}) that there is a constant $C_{t} >0$ such that
\begin{eqnarray*}
\mathbb{E}[(X_{n}(t))^{2}] & \leq & 4 \varepsilon^{2}(t\ln n)^{4}\left( (b\mathbb{E}[V_{1}^{2}])^{m_{n}} + n^{-1}\right) \tilde{m}_{n}^{2} \\
&\leq&  C_{t} \left( (\ln n)^{ \beta \log_{b} (b\mathbb{E}[V_{1}^{2}]) +4 } + n^{-1}  \tilde{m}_{n}^{2} (\ln n)^{4} \right).
\end{eqnarray*}
\noindent  Since $\mathbb{E}[V_{1}^{2}] < \mathbb{E}[V_{1}] = 1/b <1$, we see that  $\log_{b} (b\mathbb{E}[V_{1}^{2}]) < 0$. Then, we can choose $\beta$ large enough to obtain that  $(\ln n)^{-2}\mathbb{E}[(X_{n}(t))^{2}]\rightarrow 0$, as $n \rightarrow \infty$, which finishes the proof.
\end{proof}

Recall the definition of $M_{n}(t)$ in (\ref{MM}). 

\begin{lemma} \label{lemma4}
Suppose that Condition \ref{Cond1} holds. Then, for every fixed $t \in [0, \infty)$, we have that
$ (\ln n)^{-1}\mathbb{E}[M_{n}(t)] \rightarrow \mu^{-1}t$ as $n \rightarrow \infty$.
\end{lemma}

\begin{proof}
For $t \in [0, \infty)$, define 
\begin{eqnarray*}
\hat{M}_{n}(t) \coloneqq \# \left\{ v \in T_{n}^{\rm sp}: v \neq \circ \hspace*{3mm} \text{and} \hspace*{3mm} n \mathcal{L}_{v}  \geq \frac{n}{t\ln n} \right\}.
\end{eqnarray*}
\noindent Recall that we write $m_{n} =  \lfloor \beta \log_{b} \ln n \rfloor$ for some large constant $\beta > 0$. We then define
\begin{eqnarray*}
\hat{M}_{n}^{(1)}(t) \coloneqq \# \left\{ v \in T_{n}^{\rm sp}: 1 \leq d_{n}(v) \leq m_{n}  \hspace*{3mm} \text{and} \hspace*{3mm}  n\mathcal{L}_{v}  \geq \frac{n}{t\ln n} \right\}
\end{eqnarray*}
\noindent and
\begin{eqnarray*}
\hat{M}_{n}^{(2)}(t) \coloneqq \hat{M}_{n}(t) - \hat{M}_{n}^{(1)}(t) =  \# \left\{ v \in T_{n}^{\rm sp}: d_{n}(v) > m_{n}  \hspace*{3mm} \text{and} \hspace*{3mm}  n\mathcal{L}_{v}  \geq \frac{n}{t\ln n} \right\}.
\end{eqnarray*}

\noindent Let $W_{i}$, $i \geq 1$ be i.i.d.\ copies of $V_{1}$, and define $L_{k} = \prod_{i=1}^{k} W_{i}$ and $S_{k} = - \ln L_{k}$ for $k \in \mathbb{N}$. The Markov inequality implies that 
\begin{eqnarray*}
\mathbb{E}[\hat{M}_{n}^{(2)}(t)] = \sum_{k >m_{n}} b^{k} \mathbb{P}\left(n L_{k} \geq \frac{n}{t\ln n} \right) \leq (t \ln n)^{2}\sum_{k > m_{n}} b^{k} \mathbb{E}[L_{k}^{2}] = (t \ln n)^{2}\sum_{k > m_{n}} (b \mathbb{E}[V_{1}^{2}])^{k}.
\end{eqnarray*}
\noindent Hence there is a constant $C_{t} >0$ (that changes from one occurrence to the next) such that
\begin{eqnarray*}
\mathbb{E}[\hat{M}_{n}^{(2)}(t)]  \leq C_{t} (\ln n)^{2}(b \mathbb{E}[V_{1}^{2}])^{m_{n}}  \leq C_{t}  (\ln n)^{ \beta \log_{b} (b\mathbb{E}[V_{1}^{2}]) +2 }.
\end{eqnarray*}

\noindent Since $\log_{b} (b\mathbb{E}[V_{1}^{2}]) < 0$, we deduce from the previous inequality that
\begin{eqnarray} \label{eq9}
 (\ln n)^{-1}\mathbb{E}[\hat{M}_{n}^{(2)}(t)] \rightarrow 0, \hspace*{4mm} \text{as} \hspace*{2mm} n \rightarrow \infty,
\end{eqnarray}
\noindent by choosing $\beta$ large enough. 

Note that 
\begin{eqnarray*}
\mathbb{E}[\hat{M}_{n}(t)] = \sum_{k \geq 1} b^{k} \mathbb{P}\left(n L_{k} \geq \frac{n}{t\ln n} \right) = \sum_{k \geq 1} b^{k} \mathbb{P}\left(S_{k} \leq \ln (t \ln n) \right).
\end{eqnarray*}
\noindent Holmgren \cite[Lemma 2.1]{Holm2012}  
has shown that under Condition \ref{Cond1} one has that
\begin{eqnarray} \label{Ceciresult}
\sum_{k \geq 1} b^{k} \mathbb{P}\left(S_{k} \leq s \right) = (\mu^{-1} + o(1))e^{s}, \hspace*{4mm} \text{as} \hspace*{2mm} s \rightarrow \infty.
\end{eqnarray}
\noindent  Then 
\begin{eqnarray} \label{eq10}
(\ln n)^{-1}\mathbb{E}[\hat{M}_{n}(t)] =  (\ln n)^{-1}\sum_{k \geq 1} b^{k} \mathbb{P}\left(S_{k} \leq \ln (t\ln n) \right) \rightarrow \mu^{-1}t, \hspace*{4mm} \text{as} \hspace*{2mm} n \rightarrow \infty.
\end{eqnarray}
\noindent The limits (\ref{eq9}) and (\ref{eq10}) show that
\begin{eqnarray} \label{eq20}
(\ln n)^{-1}\mathbb{E}[\hat{M}_{n}^{(1)}(t)] \rightarrow \mu^{-1}t, \hspace*{4mm} \text{as} \hspace*{2mm} n \rightarrow \infty.
\end{eqnarray}
\noindent On the other hand, Lemma \ref{lemma3} and the Cauchy–Schwarz inequality imply that
\begin{eqnarray} \label{eq11}
(\ln n)^{-1}\mathbb{E}[M_{n}^{(2)}(t)] \rightarrow 0, \hspace*{4mm} \text{as} \hspace*{2mm} n \rightarrow \infty.
\end{eqnarray}
\noindent Since $M_{n}(t) = M_{n}^{(1)}(t) + M_{n}^{(2)}(t)$, the combination of the limits (\ref{eq20}) and (\ref{eq11}) imply that it is enough to check that
\begin{eqnarray} \label{eq5}
\lim_{n \rightarrow \infty} (\ln n)^{-1} \left| \mathbb{E}[M_{n}^{(1)}(t)] - \mathbb{E}[\hat{M}_{n}^{(1)}(t)] \right| = 0
\end{eqnarray}
\noindent in order to complete our proof. 

In this direction, define
\begin{eqnarray*}
Y_{n}(t) \coloneqq \# \left\{ v \in T_{n}^{\rm sp}: 1 \leq d_{n}(v) \leq m_{n}  \hspace*{3mm} \text{and} \hspace*{3mm}  {\rm Bin}(n, \mathcal{L}_{v} ) + s_{1}d_{n}(v) \geq \frac{n}{t\ln n} \right\}, 
\end{eqnarray*}
\noindent and 
\begin{eqnarray*}
\tilde{Y}_{n}(t) \coloneqq \# \left\{ v \in T_{n}^{\rm sp}: 1 \leq d_{n}(v) \leq m_{n}  \hspace*{3mm} \text{and} \hspace*{3mm}  {\rm Bin}(n, \mathcal{L}_{v} ) - sd_{n}(v) \geq \frac{n}{t\ln n} \right\}.
\end{eqnarray*}
\noindent Note that the inequality (\ref{eq4}) implies that $\mathbb{E}[\tilde{Y}_{n}(t)] \leq \mathbb{E}[M_{n}^{(1)}(t)] \leq \mathbb{E}[Y_{n}(t)]$. Then to prove (\ref{eq5}), it is enough to check that the following two limits hold:
\begin{itemize}
\item[(i)] $(\ln n)^{-1} \left| \mathbb{E}[Y_{n}(t)]  - \mathbb{E}[\hat{M}_{n}^{(1)}(t)] \right| \rightarrow 0$, as $n \rightarrow \infty$, and
\item[(ii)] $(\ln n)^{-1} \left| \mathbb{E}[\tilde{Y}_{n}(t)]- \mathbb{E}[\hat{M}_{n}^{(1)}(t)] \right| \rightarrow 0$, as $n \rightarrow \infty$.
\end{itemize} 

We only prove (i) since the proof of (ii) is analogous. Note that 
\begin{align*}
\left| \mathbb{E}[Y_{n}(t)] - \mathbb{E}[\hat{M}_{n}^{(1)}(t)] \right| & = \left| \sum_{k = 1}^{m_{n}} b^{k} \mathbb{P}\left({\rm Bin}(n, L_{k}) + s_{1}k \geq \frac{n}{t\ln n} \right)  - \sum_{k = 1}^{m_{n}} b^{k} \mathbb{P}\left( L_{k} \geq \frac{1}{t\ln n} \right) \right| \\
& \leq  \sum_{k = 1}^{m_{n}} b^{k} \mathbb{P}\left({\rm Bin}(n, L_{k}) + s_{1}k \geq \frac{n}{t\ln n}, L_{k} < \frac{1}{t \ln n} \right) \\
& \hspace*{5mm} + \sum_{k = 1}^{m_{n}} b^{k} \mathbb{P}\left({\rm Bin}(n, L_{k}) + s_{1}k < \frac{n}{t\ln n}, L_{k} \geq \frac{1}{t \ln n} \right).
\end{align*}
\noindent Denote the first term on the right-hand side by $I_{n}^{(1)}$ and the second term by $I_{n}^{(2)}$. To prove (i) we first show that $I_{n}^{(1)} = o(\ln n)$ and then that $I_{n}^{(2)} = o(\ln n)$. For $\delta_{1} \in (0,1)$, we observe that
\begin{eqnarray} \label{eq14}
I_{n}^{(1)} \leq \sum_{k = 1}^{m_{n}} b^{k} \mathbb{P}\left({\rm Bin}(n, L_{k}) + s_{1}k \geq \frac{n}{t\ln n}, L_{k} < \frac{\delta_{1}}{t \ln n} \right) + \sum_{k \geq 1} b^{k} \mathbb{P}\left( \frac{\delta_{1}}{t \ln n} \leq L_{k} < \frac{1}{t \ln n} \right).
\end{eqnarray}
\noindent On the one hand, for $1 \leq k \leq m_{n}$ and $n$ large enough, there exists a constant $C_{1} >0$ (that depends on $\delta_{1}$) such that
\begin{eqnarray*} 
\mathbb{P}\left({\rm Bin}(n, L_{k}) + s_{1}k \geq \frac{n}{t\ln n}, L_{k}  < \frac{\delta_{1}}{t \ln n} \right) & \leq &   \mathbb{P}\left({\rm Bin}(n, \delta_{1}/t \ln n) \geq \frac{n}{t\ln n} - s_{1}m_{n}\right) \nonumber \\
 & = &  \mathbb{P}\left({\rm Bin}(n, \delta_{1}/t \ln n) -\frac{\delta_{1}n}{t\ln n} \geq \frac{(1-\delta_{1})n}{t\ln n} - s_{1}m_{n}\right) \nonumber \\
 & \leq & C_{1} (\ln n)/n;
\end{eqnarray*}
\noindent to obtain the last inequality we have used Chebyshev's inequality and that the variance of a $\text{Bin}(m,q)$ random variable is $mq(1-q)$. Hence, for any $\delta_{1} \in (0,1)$ there exists a constant $C_{1}>0$ (that depends on $\delta_{1}$) such that 
\begin{eqnarray} \label{eq12}
 \sum_{k = 1}^{m_{n}} b^{k} \mathbb{P}\left({\rm Bin}(n, L_{k}) + s_{1}k \geq \frac{n}{t\ln n}, L_{k} < \frac{\delta_{1}}{t \ln n} \right) \leq C_{1} b^{m_{n}}(\ln n)/n = o(1);
\end{eqnarray}
\noindent note that the $o(1)$ does not depend on $\delta_{1}$. On the other hand, (\ref{Ceciresult}) implies that \begin{eqnarray} \label{eq13}
\lim_{n \rightarrow \infty} \frac{1}{\ln n}\sum_{k \geq 1} b^{k} \mathbb{P}\left( \frac{\delta_{1}}{t \ln n} \leq L_{k} < \frac{1}{t \ln n} \right) & = & \lim_{n \rightarrow \infty} \frac{1}{\ln n}\sum_{k \geq 1} b^{k} \mathbb{P}\left(\ln (t \ln n) < S_{k} \leq  \ln(t \ln n) - \ln \delta_{1} \right) \nonumber \\
& = & (\delta_{1}^{-1}-1)\mu^{-1}t.
\end{eqnarray}
\noindent By combining (\ref{eq12}) and (\ref{eq13}) into (\ref{eq14}), we obtain that for any $\delta_{1} \in (0,1)$, $\limsup_{n \rightarrow \infty} I_{n}^{(1)}/\ln n = (\delta_{1}^{-1}-1)\mu^{-1}t$. Thus from the arbitrariness of $\delta_{1} \in (0,1)$, we deduce that $I_{n}^{(1)} = o(\ln n)$. We complete the proof of (i) by showing that 
$I_{n}^{(2)} = o(\ln n)$. For $\delta_{2} >1$, we observe that 
\begin{eqnarray} \label{eq15}
I_{n}^{(2)} \leq \sum_{k = 1}^{m_{n}} b^{k} \mathbb{P}\left({\rm Bin}(n, L_{k}) + s_{1}k < \frac{n}{t\ln n}, L_{k} \geq \frac{\delta_{2}}{t \ln n} \right) + \sum_{k \geq 1} b^{k} \mathbb{P}\left( \frac{1}{t \ln n} \leq L_{k} < \frac{\delta_{2}}{t \ln n} \right).
\end{eqnarray}
\noindent On the one hand, for $1 \leq k \leq m_{n}$ and $n$ large enough, Chebyshev's inequality implies that there exists a constant $C_{2} >0$ (that depends on $\delta_{2}$) such that
\begin{eqnarray*} 
\mathbb{P}\left({\rm Bin}(n, L_{k}) + s_{1}k < \frac{n}{t\ln n}, L_{k} \geq \frac{\delta_{2}}{t \ln n} \right) & \leq & \mathbb{P}\left({\rm Bin}(n, \delta_{2} /t \ln n )  < \frac{n}{t\ln n}-s_{1} \right) \nonumber \\
& = & \mathbb{P}\left(\frac{\delta_{2}n}{t \ln n} - {\rm Bin}(n, \delta_{2} /t \ln n )  > \frac{(\delta_{2}-1)n}{t\ln n}-s_{1}\right) \nonumber \\
& \leq & C_{2} (\ln n)/n.
\end{eqnarray*}
\noindent Hence,  for any $\delta_{2} >1$ there exists a constant $C_{2}>0$ (that depends on $\delta_{2}$) such that 
\begin{eqnarray} \label{eq16}
 \sum_{k = 1}^{m_{n}} b^{k} \mathbb{P}\left({\rm Bin}(n, L_{k}) + s_{1}k < \frac{n}{t\ln n}, L_{k} \geq \frac{\delta_{2}}{t \ln n} \right) \leq C_{2} b^{m_{n}}(\ln n)/n = o(1).
\end{eqnarray}
\noindent The $o(1)$ does not depend on $\delta_{2}$. On the other hand, by using (\ref{Ceciresult}), we obtain that for any $\delta_{2} >1$,
\begin{eqnarray} \label{eq17}
\lim_{n \rightarrow \infty} \frac{1}{\ln n}\sum_{k \geq 1} b^{k} \mathbb{P}\left( \frac{1}{t \ln n} \leq L_{k} < \frac{\delta_{2}}{t \ln n} \right) = (1-\delta_{2}^{-1})\mu^{-1}t.
\end{eqnarray}
\noindent By combining (\ref{eq16}) and (\ref{eq17}) into (\ref{eq15}), we obtain that $\limsup_{n \rightarrow \infty} I_{n}^{(2)}/\ln n = (1-\delta_{2}^{-1})\mu^{-1}t$ and thus, the arbitrariness of $\delta_{2} > 1$ implies that $I_{n}^{(2)} = o(\ln n)$. This finishes the proof of (i).
\end{proof} 

\begin{lemma} \label{lemma5}
Suppose that Condition \ref{Cond1} holds. Then, for every fixed $t \in [0, \infty)$, we have that 
$(\ln n)^{-2}Var(M_{n}(t)) \rightarrow 0$  as $n \rightarrow \infty$. 
\end{lemma}

\begin{proof}
Write $\theta_{n} = \lfloor c \log_{b} \ln n \rfloor$, for some constant $c >0$. Fix an arbitrary small $\varepsilon >0$ 
and choose $c >0$ small enough such that  $\# \left \{ v \in T_{n}: 1 \leq d_{n}(v) < \theta_{n}  \right \} = o(\ln^{\varepsilon}n)$. (For our purpose it is enough to choose $0 < \varepsilon < 1$.) This implies that 
\begin{eqnarray} \label{eq18}
Z_{\theta_{n}} = \# \left \{ v \in T_{n}: 1 \leq d_{n}(v) < \theta_{n} \hspace*{2mm} \text{and} \hspace*{2mm} n_{v} \geq \frac{nt}{\ln n}  \right \} = o(\ln^{\varepsilon}n).
\end{eqnarray}
\noindent Let $\Omega_{\theta_{n}}$ be the $\sigma$-algebra generated by $(n_{v}: 1 \leq d_{n}(v) \leq \theta_{n})$. Note that
\begin{eqnarray*}
Var(M_{n}(t)) = \mathbb{E}[Var(M_{n}(t) | \Omega_{\theta_{n}})] + Var(\mathbb{E}[M_{n}(t) | \Omega_{\theta_{n}}]);
\end{eqnarray*}
\noindent see for instance \cite[Problem 2, Chapter 10]{Gut2013} for the variance formula. Then Lemma \ref{lemma5} follows by showing
\begin{itemize}
\item[(i)] $(\ln n)^{-2}\mathbb{E}[Var(M_{n}(t) | \Omega_{\theta_{n}})] \rightarrow 0$ as $n \rightarrow \infty$, and
\item[(ii)] $(\ln n)^{-2} Var(\mathbb{E}[M_{n}(t) | \Omega_{\theta_{n}}]) \rightarrow 0$ as $n \rightarrow \infty$.
\end{itemize}

We start with the proof of (i). For $1 \leq i \leq b^{\theta_{n}}$, let $T_{i}$ be the sub-tree of $T_{n}^{\rm sp}$ rooted at the vertex $v_{i}$ at height $\theta_{n}$ and let $n_{i}$ be number of balls stored in $T_{i}$. Recall that we write 
$m_{n} =  \lfloor \beta \log_{b} \ln n \rfloor$ for some large constant $\beta>0$. For every $t \geq 0$, we define
\begin{eqnarray*}
X_{i} \coloneqq \# \left \{ v \in T_{i}:  n_{v} \geq \frac{n}{t\ln n}  \right \} \hspace*{3mm} \text{and} \hspace*{3mm} X_{i}^{\rm c} \coloneqq \# \left \{ v \in T_{i}: d_{n}(v) \leq m_{n} \hspace*{2mm} \text{and} \hspace*{2mm} n_{v} \geq \frac{n}{t\ln n}  \right \}.
\end{eqnarray*}
\noindent Observe that we can write $M_{n}(t)=\sum_{i=1}^{b^{\theta_{n}}} X_{i} + Z_{\theta_{n}}$ and that conditioned on $\Omega_{\theta_{n}}$, $(X_{i}, 1 \leq i \leq b^{m_{n}})$ is a collection of independent random variables. Thus, it follows that
\begin{eqnarray*}
Var(M_{n}(t) \vert \Omega_{\theta_{n}}) & = & Var \left( \sum_{i=1}^{b^{\theta_{n}}} X_{i} + Z_{\theta_{n}} \Big \vert \Omega_{\theta_{n}} \right) = \sum_{i=1}^{b^{\theta_{n}}} Var(X_{i} \vert \Omega_{\theta_{n}}) \leq \sum_{i=1}^{b^{\theta_{n}}} \mathbb{E}[X_{i}^{2}\vert \Omega_{\theta_{n}} ]\\
& \leq & 2\sum_{i=1}^{b^{\theta_{n}}} \mathbb{E}[(X_{i}^{\rm c})^{2}\vert \Omega_{\theta_{n}} ] + 2\sum_{i=1}^{b^{\theta_{n}}} \mathbb{E}[(X_{i} - X_{i}^{\rm c})^{2} \vert \Omega_{\theta_{n}} ];
\end{eqnarray*}
\noindent we have used the inequality $(x+y)^{2} \leq 2x^{2} + 2y^{2}$, for $x,y \geq 0$, to obtain the last line. By the Pigeonhole principle (using that $T_{i}$ stores $n_{i}$ balls), note that for each level $\theta_{n} + j$, $j \geq 0$, we have that
\begin{eqnarray*}
\# \left \{ v \in T_{i}: d_{n}(v) =\theta_{n} + j \hspace*{2mm} \text{and} \hspace*{2mm} n_{v} \geq \frac{nt}{\ln n}  \right \} \leq \frac{n_{i}}{nt} \ln n.
\end{eqnarray*}
\noindent This implies, by recalling $\theta_{n} = \lfloor c \log_{b} \ln n \rfloor$, that there exists a constant $C>0$ (depending on $t$) such that $X_{i}^{\rm c} \leq C (n_{i}/n)(\ln n)\ln \ln n$. Then
\begin{eqnarray*}
Var(M_{n}(t) \vert \Omega_{\theta_{n}})  \leq  2C^{2} \sum_{i=1}^{b^{\theta_{n}}}  \left( \frac{n_{i} \ln n}{n} \ln \ln n \right)^{2}  + 2\sum_{i=1}^{b^{\theta_{n}}} \mathbb{E}[(X_{i} - X_{i}^{\rm c})^{2} \vert \Omega_{\theta_{n}} ].
\end{eqnarray*}
\noindent By taking expectation and applying the fact that $X_{i}-X_{i}^{\rm c} \leq M_{n}^{(2)}(t)$, Lemma \ref{lemma3} implies that 
\begin{eqnarray*}
\mathbb{E} [Var(M_{n}(t) \vert \Omega_{\theta_{n}})] \leq  C^{2} \left( \frac{\ln n}{n} \ln \ln n \right)^{2} \sum_{i=1}^{b^{\theta_{n}}}  \mathbb{E} [n_{i}^{2}]  + o(\ln^{2}n).
\end{eqnarray*}
\noindent From the inequality (\ref{eq19}), note that  $\mathbb{E} [n_{i}^{2}] \leq n^{2} (\mathbb{E}[V_{1}^{2}])^{\theta_{n}} + O(nb^{-\theta_{n}}\theta_{n}) + O(\theta_{n}^{2})$. By combining the previous two inequalities, we see that there is a constant $C^{\prime} >0$ such that 
\begin{eqnarray*}
\mathbb{E} [Var(M_{n}(t) \vert \Omega_{\theta_{n}})] & \leq &  C^{\prime} (\ln n)^{2} (\ln \ln n)^{2} (b\mathbb{E}[V_{1}^{2}])^{c\log_{b}\ln n}  + o(\ln^{2}n) \\
& = & C^{\prime} (\ln n)^{ c\log_{b}( b\mathbb{E}[V_{1}^{2}] ) +2} (\ln \ln n)^{2}  + o(\ln^{2}n).
\end{eqnarray*}
\noindent Since $\mathbb{E}[V_{1}^{2}] < 1/b < 1$, we see that $\log_{b}( b\mathbb{E}[V_{1}^{2}] )<0$. It then follows that $\mathbb{E} [Var(M_{n}(t) \vert \Omega_{\theta_{n}})] = o(\ln^{2}n)$ which shows (i).

Next, we prove point (ii). Recall that conditioned on $\Omega_{\theta_{n}}$, $(X_{i}, 1 \leq i \leq b^{m_{n}})$ is a sequence of independent random variables. Moreover, for $1 \leq i \leq b^{m_{n}}$, the sub-tree $T_{i}$ is a split tree with $n_{i}$ balls. Then Lemma \ref{lemma4} implies that
\begin{eqnarray*}
\mathbb{E}[X_{i} \vert \Omega_{\theta_{n}}] = \mu^{-1}t  \frac{n_{i}}{n} \ln n + o \left(\frac{n_{i} \ln n}{n} \right), \hspace*{4mm} \text{for} \hspace*{2mm} 1 \leq i \leq b^{\theta_{n}}.
\end{eqnarray*}
\noindent Thus,
\begin{eqnarray*}
\mathbb{E}[M_{n}(t) \vert \Omega_{\theta_{n}}]  =  \mathbb{E} \left[ \sum_{i=1}^{b^{\theta_{n}}} X_{i} + Z_{\theta_{n}} \Big \vert \Omega_{\theta_{n}} \right] = \sum_{i=1}^{b^{\theta_{n}}} \mu^{-1}t \frac{n_{i}}{n}\ln n +  o \left(\frac{n_{i} \ln n}{n} \right) + \mathbb{E}[Z_{\theta_{n}} \vert \Omega_{\theta_{n}}].
\end{eqnarray*}
\noindent By our choice of $c>0$ in (\ref{eq18}), we see that $Var(\mathbb{E}[M_{n}(t) \vert \Omega_{\theta_{n}}]) = Var(\mu^{-1}t\ln n+ o(\ln n)) = o(\ln^{2}n)$. This proves (ii) and concludes the proof of Lemma \ref{lemma5}. 
\end{proof}

Finally, we can now easily prove Proposition \ref{lemma1}.

\begin{proof}[Proof of Proposition \ref{lemma1}]
This is a simple consequence of Lemma \ref{lemma4}, Lemma \ref{lemma5} and Chebyshev's inequality.
\end{proof}

\section{Proof of Theorem \ref{Teo4}}  \label{LatticeCase}

In this section, we prove Theorem \ref{Teo4} which follows from a simple adaptation of the proof of Theorem \ref{Teo1}. Therefore, we only provide enough details to convince the reader that the proof can be carried out as in Theorem \ref{Teo1}. Indeed, we just need a lattice version of Proposition \ref{lemma1}.  

For $t \in [0, \infty)$, recall the definition of $M_{n}(t)$ in (\ref{MM}).  For every $\varrho \in [0,1)$ and $x \geq 0$, consider the function $\varphi(x) = 1/x$ and define the measure $\Xi_{\varrho}^{\ast}$ by $\Xi_{\varrho}^{\ast}(A) = \Xi_{\varrho} (\varphi^{-1}(A))$, for all measurable subsets $A \subset \mathbb{R}_{+}$. 

\begin{proposition} \label{ExtraPro1}
Let $T_{n}^{\rm sp}$ be a split tree such that $\mathbb{P}(V_{1} = 1) = \mathbb{P}(V_{1} = 0) = 0$ and that $-\ln V_{1}$ is lattice with span $a>0$. Suppose also that $p_{n}$ fulfills (\ref{eq1}) and that $n \rightarrow \infty$ such that $\{a^{-1} \ln \ln n\} \rightarrow \varrho \in [0,1)$. Then, for every fixed $t \in [0, \infty)$, we have that
$ (\ln n)^{-1}M_{n}(t) \overset{\mathbb{P}}{\longrightarrow} \mu^{-1}\Xi_{\varrho}^{\ast}((0,t])$, as $n \rightarrow \infty$.
\end{proposition}

\begin{proof}
A close inspection of the proof of Proposition \ref{lemma1} shows that we only need a lattice version of the result in (\ref{Ceciresult}). Let $W_{i}$, $i \geq 1$ be i.i.d.\ copies of $V_{1}$, and define $L_{k} = \prod_{i=1}^{k} W_{i}$ and $S_{k} = - \ln L_{k}$ for $k \in \mathbb{N}$. Since $-\ln V_{1}$ is lattice with span $a >0$, it follows from \cite[Lemma 1]{Ber2019} that
\begin{eqnarray*} 
\sum_{k \geq 1} b^{k} \mathbb{P}\left(S_{k} \leq a \lfloor s \rfloor \right) = \left(\frac{a}{\mu} \frac{1}{1-e^{-a}} + o(1) \right)e^{ a \lfloor s \rfloor}, \hspace*{3mm} \text{as} \hspace*{2mm} s \rightarrow \infty.
\end{eqnarray*}

\noindent In particular, for $t \in [0, \infty)$,
\begin{eqnarray*}
(\ln n)^{-1}\sum_{k \geq 1} b^{k} \mathbb{P}\left(S_{k} \leq \ln (t\ln n) \right) = (\ln n)^{-1}\sum_{k \geq 1} b^{k} \mathbb{P}\left(S_{k} \leq  a\lfloor a^{-1} \ln (t \ln n) \rfloor \right)  \rightarrow  \mu^{-1}\Xi_{\varrho}^{\ast}((0,t]), 
\end{eqnarray*}

\noindent as $n \rightarrow \infty$, which corresponds to (\ref{eq10}) in the proof of Proposition \ref{lemma1}. Therefore, Proposition \ref{ExtraPro1} follows along the lines of the proof of Proposition \ref{lemma1}. Details are left to the reader.
\end{proof}

\begin{proof}[Proof of Theorem \ref{Teo4}]
In \cite[Lemma 2]{Ber2019}, we have proved the first claim in Theorem \ref{Teo4}.
The second claim follows from the same argument as in the proof of the second part of Theorem \ref{Teo1} by using Proposition \ref{ExtraPro1}. We leave the details to the interested reader. 
\end{proof}

\section{Proof of Theorem \ref{Teo3}} \label{sec5}
 
In this section, we point out that the approach used in the proof of Theorem \ref{Teo1} can also be applied to study percolation on complete $d$-regular trees to prove Theorem \ref{Teo3}. Recall that we consider a rooted complete regular $d$-ary tree $T_{h}^{\rm d}$ of height $h \in \mathbb{N}$, where $d \geq 2$ is some integer, 
and perform Bernoulli  bond-percolation with parameter $q_{h}$ fulfilling (\ref{per}). Recall also that $\log_{d} x = \ln x/\ln d$ denotes the logarithm with base $d$ of $x>0$.

For $t \in [0, \infty)$, let $H_{h}(t)$ be the number of edges at height less or equal to $\lfloor \log_{d}ht \rfloor$ that has been removed. More precisely, let $\tau_{i}(h)$ be the $i$-th smallest height at which an edge has been removed. Then
\begin{eqnarray*}
H_{h}(t) \coloneqq \sum_{i \geq 1} \mathds{1}_{\{\tau_{i}(h) \leq \log_{d}ht \}} =  \sum_{i \geq 1} \mathds{1}_{\{h^{-1}d^{\tau_{i}(h)} \leq t \}}.
\end{eqnarray*}
\noindent For every $\rho \in [0,1)$ and $x \geq 0$, consider the function $\varphi(x) = 1/x$ and define the measure $\Lambda_{\rho}^{\ast}$ by $\Lambda_{\rho}^{\ast}(A) = \Lambda_{\rho} (\varphi^{-1}(A))$  for all measurable subsets $A \subset \mathbb{R}_{+}$. In particular, 
\begin{eqnarray*}
\Lambda_{\rho}^{\ast}((0,x]) = \Lambda_{\rho}([1/x, \infty)) =  d^{-\rho +\lfloor \rho +\log_{d}x \rfloor +1}/(d-1).
\end{eqnarray*}

\begin{proposition} \label{Pro4}
Suppose that $\{ \log_{d}h \} \rightarrow \rho \in [0,1)$ as $h \rightarrow \infty$ and that $q_{h}$ fulfills (\ref{per}). Then, the following convergence holds in the sense of weak convergence of finite dimensional distributions,
\begin{eqnarray*}
(H_{h}(t), t \geq 0) \overset{d}{\longrightarrow} (H(t), t \geq 0),\hspace*{5mm} \text{as} \hspace*{2mm} h \rightarrow \infty,
\end{eqnarray*}
\noindent where $(H(t), t \geq 0)$ is a Poisson process with intensity $c\Lambda_{\rho}^{\ast}({\rm d}x)$.
\end{proposition} 

\begin{proof}
Note that $(H_{h}(t), t \geq 0)$ has independent increments; since $H_{h}(t) - H_{h}(s)$ is the number of edges removed between height $\lfloor\log_{d}(hs) \rfloor$ and $\lfloor\log_{d}(ht) \rfloor$, for $0 \leq s \leq t$. Furthermore, 
\begin{eqnarray*}
H_{h}(t) - H_{h}(s) \stackrel{d}{=} {\rm Bin} \left((d^{\lfloor\log_{d}(ht) \rfloor +1 } - d^{\lfloor\log_{d}(hs) \rfloor +1})/(d-1), 1-q_{h}  \right), 
\end{eqnarray*}
\noindent On the one hand, $1-q_{h} \rightarrow 0$ as $h \rightarrow \infty$. On the other hand, 
\begin{eqnarray*}
(1-q_{h}) \frac{d^{\lfloor\log_{d}(ht) \rfloor +1}  - d^{\lfloor\log_{d}(hs) \rfloor +1}}{d-1} = (c+o(1)) \frac{d^{-\{ \log_{d} h\} + \lfloor \{ \log_{d} h\} + \log_{d} t \rfloor +1}  - d^{-\{ \log_{d} h\} +\lfloor \{\log_{d}h \} + \log_{d} s \rfloor +1}}{d-1}.
\end{eqnarray*}
\noindent Then, $H_{h}(t) - H_{h}(s) \overset{d}{\longrightarrow} {\rm  Poisson}(c\Lambda_{\rho}^{\ast}((s,t]))$, as $h \rightarrow \infty$, which clearly implies our claim. 
\end{proof}

\begin{corollary} \label{cor3}
Suppose that $\{ \log_{d}h \} \rightarrow \rho \in [0,1)$ as $h \rightarrow \infty$ and that $q_{h}$ fulfills (\ref{per}). Then, for every fixed $i \in \mathbb{N}$, 
\begin{eqnarray*}
( hd^{-\tau_{1}(h)}, \dots, hd^{-\tau_{i}(h)} ) \overset{d}{\longrightarrow} ({\rm x}_{1}, \dots, {\rm x}_{i}), \hspace*{5mm} \text{as} \hspace*{2mm} h \rightarrow \infty,
\end{eqnarray*}
\noindent where ${\rm x}_{1} \geq {\rm x}_{2} \geq \cdots$ are the atoms of a Poisson process on $(0, \infty)$ with intensity $c \Lambda_{\rho}({\rm d} x)$. 
\end{corollary}

\begin{proof}
Note that the sequence $h^{-1}d^{\tau_{1}(h)}  \leq h^{-1}d^{\tau_{2}(h)}  \leq \cdots$ are the occurrence times of the counting process $(H_{h}(t), t \geq 0)$. Then Proposition \ref{Pro4} implies that  for every fixed $i \in \mathbb{N}$, 
\begin{eqnarray*}
( h^{-1}d^{\tau_{1}(h)}, \dots, h^{-1}d^{\tau_{i}(h)} ) \overset{d}{\longrightarrow} ({\rm y}_{1}, \dots, {\rm y}_{i}), \hspace*{5mm} \text{as} \hspace*{2mm} h \rightarrow \infty,
\end{eqnarray*}
\noindent where ${\rm y}_{1} \leq {\rm y}_{2} \leq \cdots$ are the occurrence times of the Poisson process $(H(t), t \geq 0)$. Therefore, our claim follows from the continuous mapping theorem (\cite[Theorem 2.7, Chapter 1]{Billi1999}) and basic properties of Poisson processes (\cite[Proposition 3.7, Chapter 3]{Res1987}).
\end{proof}

Let ${\bf e}_{i}(h)$ be the edge with the $i$-th smallest height that has been removed and ${\bf v}_{i}(h)$ the endpoint (vertex) of ${\bf e}_{i}(h)$ that is further away from the root of $T_{h}^{ \rm d}$. Let $T_{i}(h)$ be the subtree of $T_{h}^{ \rm d}$ that is rooted at ${\bf v}_{i}(h)$ and denote by $\tilde{G}_{i}$ the size (number of vertices) of the root-cluster of $T_{i}(h)$ after performing percolation (where here of course root means ${\bf v}_{i}(h)$). We also write $\tilde{G}_{i}^{\ast}$ for the size (number of vertices) of the largest cluster of $T_{h}(i)$ that does not contain the root. 

\begin{proposition} \label{Pro5}
Suppose that $\{ \log_{d}h \} \rightarrow \rho \in [0,1)$ as $h \rightarrow \infty$ and that $q_{h}$ fulfills (\ref{per}). For every fixed $i \in \mathbb{N}$, we have that
\begin{eqnarray*}
\tilde{G}_{k}^{\ast} = o_{\rm p}(h^{-1}d^{h}) \hspace*{5mm} \text{for} \hspace*{2mm} k =1, \dots, i.
\end{eqnarray*}
\noindent Furthermore, we have that
\begin{eqnarray*}
( d^{-h+\tau_{1}(h)} \tilde{G}_{1}, \dots, d^{-h+\tau_{i}(h)} \tilde{G}_{i} ) \overset{\mathbb{P}}{\longrightarrow} \frac{d}{d-1}(e^{-c}, \dots, e^{-c}), \hspace*{5mm} \text{as} \hspace*{2mm} h \rightarrow \infty.
\end{eqnarray*}
\end{proposition}

\begin{proof}
Note that $T_{h}(i)$ is a $d$-ary tree of height $h - \tau_{i}(h)$ with  $n_{i}(h) = (d^{h - \tau_{i}(h)+1}-1)/(d-1)$ vertices. Note also that Corollary \ref{cor3} implies that $\tau_{i}(h)/h \overset{\mathbb{P}}{\longrightarrow} 0$, and then $h - \tau_{i}(h) \overset{\mathbb{P}}{\longrightarrow} \infty$, as $h \rightarrow \infty$. Then, one sees that the percolation parameter $q_{h}$ in (\ref{per}) corresponds precisely to the supercritical regime in $T_{i}(h)$. Therefore, our claim follows from \cite[Theorem 1 and Proposition 1]{Be1} by verifying that the hypotheses ($H_{k}$) and ($H_{k}^{\prime}$) there hold for every $k \in \mathbb{N}$ with $\ell(n_{i}(h)) = \ln (n(h))$ and $\xi_{k} \equiv 1/ \ln d$. 
\end{proof}

We now have all the ingredients to prove Theorem \ref{Teo3}. As in the proof of Theorem \ref{Teo2}, we only provide enough details to convince the reader that everything can be carried out as in the proof of Theorem \ref{Teo1}.

\begin{proof}[Proof of Theorem \ref{Teo3}]
The first claim has been shown in \cite[Section 3]{Be1}, and thus, we only prove the second one. The same argument as in the proof of Corollary \ref{cor2} shows, by using Corollary \ref{cor3} and Proposition \ref{Pro5}, that for every fixed $i \in \mathbb{N}$,
\begin{eqnarray*}
( hd^{-h}\tilde{G}_{1}, \dots, hd^{-h}\tilde{G}_{i} ) \overset{d}{\longrightarrow} ({\rm x}_{1}, \dots, {\rm x}_{i}), \hspace*{5mm} \text{as} \hspace*{2mm} h \rightarrow \infty,
\end{eqnarray*}
\noindent where ${\rm x}_{1} \geq {\rm x}_{2} \geq \cdots$ are the atoms of a Poisson process on $(0, \infty)$ with intensity $c \frac{d}{d-1}e^{-c}  \Lambda_{\rho}({\rm d} x)$. 

For $\ell \in  \mathbb{N}$, denote by $ \bar{G}_{1, \ell} \geq \bar{G}_{2, \ell} \geq \cdots \geq \bar{G}_{\ell, \ell} $ the rearrangement in decreasing order of the $\bar{G}_{i}$ for $i =1, \dots, \ell$. As in the proof of Lemma \ref{lemma2}, one can show that for every fixed $i \in \mathbb{N}$,
\begin{eqnarray*}
\lim_{\ell \rightarrow \infty} \liminf_{h \rightarrow \infty} \mathbb{P} \left(\tilde{G}_{k, \ell}= G_{k} \hspace*{3mm} \text{for every} \hspace*{2mm} k =1, \dots, i    \right) = 1;
\end{eqnarray*}
\noindent  one only needs to note that $\Lambda_{\rho}([x, \infty)) \asymp x^{-1}$, for $x > 0$, and thus, a Poisson process on $(0,\infty)$ with intensity $c \frac{d}{d-1}e^{-c}  \Lambda_{\rho}({\rm d} x)$ has infinitely many atoms\footnote{For every pair of functions $f,g >0$, we write $f \asymp g$ if there exists a positive real number $c$ such that $cf(x) \leq g(x) \leq f(x)/c$ for all $x$.}. 

Finally, a combination of the previous two facts conclude with the proof of Theorem \ref{Teo3}.
\end{proof}

\paragraph{Acknowledgements.}
This work is supported by the Knut and Alice Wallenberg
Foundation, a grant from the Swedish Research Council and The Swedish Foundations' starting grant from Ragnar S\"oderbergs Foundation.


\providecommand{\bysame}{\leavevmode\hbox to3em{\hrulefill}\thinspace}
\providecommand{\MR}{\relax\ifhmode\unskip\space\fi MR }
\providecommand{\MRhref}[2]{%
  \href{http://www.ams.org/mathscinet-getitem?mr=#1}{#2}
}
\providecommand{\href}[2]{#2}

\end{document}